\documentclass[reqno]{amsart}
\usepackage{amsmath,amsthm,amssymb,mathrsfs,stmaryrd,eucal,mathbbol}
\usepackage{pgfplots}
\usepackage{tikz}
\usepackage[all,cmtip]{xy}
\usepackage{amsaddr,hyperref}

\newcommand\void[1]{}

\newcommand{\be}{\begin{equation}}
\newcommand{\ee}{\end{equation}}
\newcommand{\bnu}{\begin{enumerate}}
\newcommand{\enu}{\end{enumerate}}

\numberwithin{equation}{section}

\newcommand{\CA}{\mathcal{A}}
\newcommand{\CB}{\mathcal{B}}
\newcommand{\CC}{\mathcal{C}}
\newcommand{\CD}{\mathcal{D}}
\newcommand{\CE}{\mathcal{E}}
\newcommand{\CF}{\mathcal{F}}

\newcommand{\CM}{\mathcal{M}}
\newcommand{\CN}{\mathcal{N}}
\newcommand{\CO}{\mathcal{O}}

\newcommand{\CW}{\mathcal{W}}
\newcommand{\CX}{\mathcal{X}}

\newcommand{\FZ}{\mathfrak{Z}}

\newcommand{\bk}{\mathbf{k}}

\DeclareMathOperator{\Hom}{Hom}

\DeclareMathOperator{\Img}{Im}
\DeclareMathOperator{\Id}{Id}

\DeclareMathOperator{\ev}{ev}
\DeclareMathOperator{\coev}{coev}
\DeclareMathOperator{\chara}{char}
\DeclareMathOperator{\Rep}{Rep}

\DeclareMathOperator{\Fun}{Fun}

\DeclareMathOperator{\Alg}{Alg}

\DeclareMathOperator{\LMod}{LMod}
\DeclareMathOperator{\RMod}{RMod}
\DeclareMathOperator{\BMod}{BMod}

\newcommand{\adj}[4]{\xymatrix{ #1 \ar@<.5ex>[r]^-{#3} & #2 \ar@<.5ex>[l]^-{#4}}}

\newcommand{\Cat}{{\mathcal{C}at}}

\newcommand{\one}{\mathbf{1}}
\newcommand{\rev}{\mathrm{rev}}
\newcommand{\op}{\mathrm{op}}

\newcommand{\mtc}{\mathcal{MT}\mathrm{en}}
\newcommand{\bmtc}{\mathcal{BT}\mathrm{en}}
\newcommand{\mfus}{\mathcal{MF}\mathrm{us}}
\newcommand{\bfus}{\mathcal{BF}\mathrm{us}}
\newcommand{\cl}{\mathrm{cl}}
\newcommand{\ind}{\mathrm{ind}}

\newtheorem{thm}{Theorem}[subsection]
\newtheorem{lem}[thm]{Lemma}
\newtheorem{prop}[thm]{Proposition}
\newtheorem{cor}[thm]{Corollary}

\newtheorem{prop-defn}[thm]{Proposition-Definition}

\theoremstyle{definition}
\newtheorem{defn}[thm]{Definition}

\newtheorem{exam}[thm]{Example}
\newtheorem{rem}[thm]{Remark}

\theoremstyle{remark}
\title{The symmetry enriched center functor is fully faithful}
\author{\small Long Sun}
\address{\rm \footnotesize Department of Mathematics, \\ Peking University, Beijing, 100871, China}
\email{1901110006@pku.edu.cn}
\begin{document}
	
\maketitle
	
%%%%%%%%%%%%%%%%%%%%%%%%%%%%%%%%%%%%%%%%%%%%%%%%%%%%%%%%
	
\begin{abstract}
		
In this work, inspired by some physical intuitions, we define a series of symmetry enriched categories to describe symmetry enriched topological (SET) orders, and define a new tensor product, called the relative tensor product, which describes the stacking of 2+1D SET orders. Then we choose and modify the domain and codomain categories, and manage to make the Drinfeld center a fully faithful symmetric monoidal functor. 
It turns out that this functor, named the symmetry enriched center functor, provides a precise and rather complete mathematical formulation of the boundary-bulk relation of symmetry enriched topological (SET) orders.
We also provide another description of the relative tensor product via a condensable algebra.

\end{abstract}
	
%%%%%%%%%%%%%%%%%%%%%%%%%%%%%%%%%%%%%%%%%%%%%%%%%%%%%%%%
	
\tableofcontents
	
%%%%%%%%%%%%%%%%%%%%%%%%%%%%%%%%%%%%%%%%%%%%%%%%%%%%%%%%
	
\section{Introduction} \label{sec:intro}

	Recently, the study of 2d rational conformal field theories (\cite{FFRS1,FFRS2,DKR1,DKR2}) and that of topological orders in condensed matter physics (\cite{LW,KK,Ko}) have become very popular and crucial in mathematical physics.
	The physical intuitions from these theories have directly inspired many new mathematical results (see for examples \cite{DMNO,DNO,KZ1,LKW1,KWZ2,GVR,VR,KLWZZ,KZ3}).
	
	Among these works, in \cite{KZ1}, it was shown that the Drinfeld center can be made functorial and fully faithful if we choose proper domain and codomain categories. 
	More precisely, one can choose the domain category as the category $\mfus_k^\ind$ of indecomposable multi-fusion categories over an algebraically closed field $k$ of characteristic zero with semisimple bimodules as morphisms and the codomain category as the category $\bfus_k^\cl$ of braided fusion categories with closed multi-fusion bimodules as morphisms. Then, for $\CC,\CD \in \mfus_k^\ind$ and a $\CC$-$\CD$-bimodule $\CM$, it was proved that the assignment
	\begin{equation} \label{eq:center-functor}
	(\CC \xrightarrow{\CM} \CD) \quad \mapsto \quad (\FZ(\CC) \xrightarrow{\FZ^{(1)}(\CM) := \Fun_{\CC|\CD}(\CM,\CM)} \FZ(\CD)),
	\end{equation}
	where $\FZ(\CC)$ denotes the Drinfeld center of $\CC$ and $\Fun_{\CC|\CD} (\CM,\CM)$ denotes the category of $\CC$-$\CD$-bimodule functors, gives a well-defined fully faithful symmetric monoidal functor.

	In fact, this Drinfeld center functor provides a precise and complete mathematical description of the boundary-bulk relation of 2+1D anomaly-free topological orders with gapped boundaries (see \cite{KWZ1,KWZ2}).
	In particular, $\FZ^{(1)}(\CM)$ describes a 1+1D gapped domain wall between the two 2+1D topological phases described by $\FZ(\CC)$ and $\FZ(\CD)$.
	
	In the theory of  topological orders, a topological order with symmetry is called a symmetry enriched topological (SET) order (see for examples \cite{LG,CGLW,BBCW,LKW1}). 
	Inspired by the boundary-bulk relation of 2+1D (anomaly-free) SET orders with gapped boundaries, we give a generalization of the work in \cite{KZ1} to describe the symmetry enriched case, more precisely, to find categorical descriptions of SET orders and give a notion of the symmetry enriched center and make it functorial by choosing proper domain and codomain categories. 
	
	It was illustrated in \cite{LKW1} that the symmetry can be characterized by a symmetric fusion category $\CE$ over a field $k$. These three authors in \cite{LKW1} also showed that one can use a unitary fusion category over $\CE$ to describe the excitations of a 1+1D SET order and a triple $(\CE,\CC,\CM)$ to characterize a 2+1D anomoly-free SET order, where $\CC$ is a unitary modular tensor category over $\CE$ and $\CM$ is a minimal modular extension of $\CC$ (see \cite{Mu1,GVR,VR} for definition). Indeed, $\CE$ describes the symmetry charges, and $\CC$ describes the local topological excitations, and $\CM$ describes the topological excitations after gauging the symmetry (see \cite{LKW1,KLWZZ} for details).
	
	In our work, we soften the terms and study more general cases.
	We first set the domain category as the category $\mtc_{/\CE}$ of multi-tensor categories fully faithful over $\CE$ (see Definition \ref{def:tc_E}) with finite bimodules over $\CE$ (see Definition \ref{def:bim_E}) as morphisms, and set the codomain category as the category $\bmtc_\CE$ of braided tensor categories fully faithful containing $\CE$ (see Definition \ref{def:btccE}) with morphisms given by monoidal bimodules containing $\CE$ (see Definition \ref{def:mbimcE}). 
	Naturally, we choose the tensor product on the monoidal category $\mtc_{/\CE}$ as the Deligne's tensor product over $\CE$, $\boxtimes_\CE$ (see \cite{DSS} for definition), which coincides with the stacking of 1+1D SET orders. 
	Nevertheless, things are becoming more interesting---when we consider the tensor product of braided tensor categories containing $\CE$ which characterizes the stacking of 2+1D SET orders, we find that $\boxtimes_\CE$ does not work! 

	How to properly construct a new tensor product on $\bmtc_\CE$?
	This is one of the main obstacles to making the Drinfeld center functorial and monoidal.

	Inspired by the intuition from physics, we fix our eyes on a structure on those categories, named $\CE$-module braidings (see Definition \ref{def:br}), and then define the relative tensor product over $\CE$ (see Definition \ref{def:rtp}) by involving this structure. In Section \ref{sec:rtensor_localm}, it is proved that our definition of the relative tensor product over $\CE$ coincides with that in \cite{LKW1}, which is used to characterize the stacking of 2+1D SET orders. 
	Then, we show that for $\CC,\CD \in \mtc_{/\CE}$ and an $\CC$-$\CD$-bimodule $\CM$ over $\CE$, the same assignment as the usual case
	$$
	(\CC \xrightarrow{\CM} \CD) \quad \mapsto \quad (\FZ(\CC) \xrightarrow{\FZ^{(1)} (\CM) := \Fun_{\CC|\CD} (\CM,\CM)} \FZ(\CD)),
	$$
	gives a well-defined symmetric monoidal functor. 

	Moreover, when we restrict the domain to the category $\mfus_{/\CE}$ of multi-fusion categories fully faithful over $\CE$ (see Definition \ref{def:tc_E}) with semisimple bimodules over $\CE$ as morphisms, and respectively, for the codomain category, choose the category $\bfus_\CE^\cl$ of nondegenerate braided fusion categories fully faithful containing $\CE$ (see Definition \ref{def:btccE}) with closed bimodules containing $\CE$ as morphisms, we prove that the center functor (or symmetry enriched center functor) is made fully faithful. These necessary conditions of the domain and codomain categories are inspired by the physical intuition of the boundary-bulk relation of SET orders, and the main method we used is to convert to the usual case (without symmetry) and check the additional conditions involved by the symmetry.

	The fully-faithfulness of this symmetry enriched center functor immediately implies two results: (1) two fusion categories over $\CE$ are Morita equivalent over $\CE$ if and only if their Drinfeld centers are braided monoidally equivalent respecting $\CE$ (see Definition \ref{def:bmfE}); (2) there is a group isomorphism between the group of the equivalence classes of semisimple invertible $\CC$-$\CC$-bimodules over $\CE$ and the group of the equivalence classes of braided monoidal auto-equivalences respecting $\CE$ of $\FZ(\CC)$. Both results are the generalizations of the ones for the usual case (without symmetry).

	An outline of this paper is as follows.
	In Section \ref{sec:pre}, we review relevant results in tensor categories over $k$ and set the notations along the way. 
	In Section \ref{sec:SEC}, we introduce a series of symmetry enriched categories, especially including finite monoidal categories over $\CE$ and finite braided monoidal categories containing $\CE$, and define a new tensor product between the latter, named the relative tensor product over $\CE$, via considering a structure named $\CE$-module braidings, and prove some of their properties. These notions play vital roles in our construction of the target category of the symmetry enriched center functor. We devote Section \ref{sec:rtensor_localm} to giving another description of the relative tensor product over $\CE$ and prove that it coincides with the one defined in \cite{LKW1}. 
	In Section \ref{sec:SE-center}, we construct our domain category and target category with symmetry enriched in $\CE$ and equip with proper tensor products, and then finally make the notion of the center functorial and fully faithful. 
	In Section \ref{sec:TO}, we briefly explain the motivation of our work and give the physical meaning of definitions and theorems in Section \ref{sec:SEC} and \ref{sec:SE-center}, especially Theorem \ref{thm:fully-faithful_E}.

~\\
\noindent
\textbf{Acknowledgement.}
	We are grateful for Prof. Hao Zheng's enlightening guidance of this research. And we are also thankful to Prof. Liang Kong, Dr. Zhihao Zhang, Dr. Xiaoxue Wei, Dr. Ansi Bai, Dr. Hao Xu and Dr. Shengyu Yang for many useful discussions and suggestions. LS would like to thank Shenzhen Institute for Quantum Science and Engineering, Southern University of Science and Technology for the invitation. Part of this work was carried out while this author was visiting Shenzhen Institute for Quantum Science and Engineering. The author is supported by NSFC under Grant No. 11871078. \\
	
\noindent
\textbf{Declaration.}	
	The author has no competing interests to declare that are relevant to the content of this article.

%%%%%%%%%%%%%%%%%%%%%%%%%%%%%%%%%%%%%%%%%%%%%%%%%%%%%%%%

\section{Finite monoidal categories} \label{sec:pre}

	In this section, we review some basic notions of finite monoidal categories, collect or prove a few results that are important to this paper, and set our notations along the way.

%%%%%%%%%%%%%%%%%%%%%%%%%%%%%%%%%%%%%%%%%%%%%%%%%%%%%%%%

\subsection{Finite monoidal categories and finite modules}
	
	Let $k$ be a field. We denote by $\bk$ the category of finite dimensional vector spaces over $k$. 
	
	For a monoidal category $\CC$, we use $\CC^\rev$ to denote the same category as $\CC$ but equipped with the reversed tensor product $\otimes^\rev: \CC \times \CC \to \CC$ defined by $a \otimes^\rev b := b \otimes a$.
	
	A \textit{finite category} over $k$ is a $k$-linear category $\CC$ that is equivalent to the category of modules $\RMod_A(\bk)$ for some finite dimensional algebra $A$ over $k$. We omit "over $k$" if it is clear from the context.
	We say that $\CC$ is semisimple if the defining algebra $A$ is semisimple.
	
	There is an intrinsic description of a finite category: a $k$-linear abelian category such that it is locally finite (i.e. all morphism spaces are finite dimensional, every object has finite length) and has finitely many simple objects, each of which has a projective cover. (See \cite[Definition 1.8.6]{EGNO}.)	
	
	Given two finite categories $\CC$ and $\CD$, the \textit{Deligne's tensor product} $\CC \boxtimes \CD$ is a finite category with a bilinear bifunctor which is right exact in both variables and satisfies a universal property (see \cite{De1,Ke} for definition). Note that $\bk \boxtimes \CC \simeq \CC \simeq \CC \boxtimes \bk$. 
		
	A \textit{finite monoidal category} over $k$ is a monoidal category $\CC$ such that $\CC$ is a finite category over $k$ and the tensor product $\otimes: \CC \times \CC \to \CC$ is $k$-bilinear and right exact in each variable. A \textit{finite multi-tensor category}, or \textit{multi-tensor category} for short, is a rigid finite monoidal category. A \textit{tensor category} is a multi-tensor category with a simple tensor unit. A multi-tensor category is \textit{indecomposable} if it is neither zero nor the direct sum of two nonzero multi-tensor categories. A \textit{multi-fusion category} is a semisimple multi-tensor category, and a \textit{fusion category} is a multi-fusion category with a simple tensor unit. Note that any fusion category $\CC$ contains a trivial fusion subcategory consisting of multiples of $\one_\CC$. We directly identify this subcategory with $\bk$.
	
	Let $\CC$ and $\CD$ be finite monoidal categories. A \textit{finite left $\CC$-module} $\CM$ (also denoted as $_\CC\CM$) is a left $\CC$-module such that $\CM$ is a finite category and the $\CC$-action $\odot: \CC \times \CM \to \CM$ is $k$-bilinear and right exact in each variable (see \cite{CFS,O}). We say that $\CM$ is \textit{indecomposable} if it is neither zero nor the direct sum of two nonzero finite left $\CC$-modules. The notions of a \textit{finite right $\CD$-module} (denoted as $\CM_\CD$) and a \textit{finite $\CC$-$\CD$-bimodule} (denoted as $_\CC\CM_\CD$) are defined similarly. Namely, a right $\CD$-module category is the same thing as a left $\CD^\rev$-module category, and a $\CC$-$\CD$-bimodule is exactly a left $\CC \boxtimes \CD^\rev$-module category. In this way, the notions of a left and a right module are the special cases of bimodule categories, e.g. a left $\CC$-module category can be seen as a $\CC$-$\bk$-bimodule.
	
	Let $\CC$,$\CD$ be finite monoidal categories and $_\CC\CM_\CD,_\CC\CN_\CD$ be finite $\CC$-$\CD$-bimodules. A \textit{$\CC$-$\CD$-bimodule functor} $F: \CM \to \CN$ is a $k$-linear functor equipped with two isomorphisms $S^{F|L}_{c,x}: c \odot F(x) \to F(c \odot x)$ and $S^{F|R}_{x,d}: F(x) \odot d \to F(x \odot d)$ satisfying some natural axioms (see \cite{O}).
	We use the notations $S^{F|L}$ and $S^{F|R}$ (or $S^F$ for short if there is only one module structure on some side) to denote the left and right module isomorphisms of a module functor $F$ throughout this paper.
	We use $\Fun_{\CC|\CD}(\CM,\CN)$ to denote the category of right exact $\CC$-$\CD$-bimodule functors from $\CM$ to $\CN$. If $\CC = \bk$ (resp. $\CD = \bk$ or $\CC = \CD = \bk$), we abbreviate $\Fun_{\CC|\CD}(\CM,\CN)$ to $\Fun_{\CD^\rev}(\CM,\CN)$ (resp. $\Fun_\CC(\CM,\CN)$ or $\Fun(\CM,\CN)$).	
	
	For an object $a \in \CC$, we use $a^L$ or $a^R$ to denote the left dual or the right dual of $a$ respectively if exists. 
	When $\CC$ is rigid, taking the left dual determines a monoidal equivalence 
	$$\delta^{L} : \CC^\rev \to \CC^\op, x \mapsto x^L,$$ 
	and taking the right dual determines a monoidal equivalence 
	$$\delta^{R} : \CC^\rev \to \CC^\op, x \mapsto x^R.$$ 
		
	Given a left $\CC$-module $\CM$, the opposite category $\CM^\op$ admits two natural right $\CC$-module structures induced by $\delta^{L}$ and $\delta^{R}$, which are denoted by $\CM^{\op|L}$ and $\CM^{\op|R}$, respectively. More precisely, for $x \in \CM, a \in \CC$, we have
	$$
	\begin{aligned}
	(\CM^{\op|L},\odot^L)&: &x \odot^L a := a^L \odot x , \\
	(\CM^{\op|R},\odot^R)&: &x \odot^R a := a^R \odot x .
	\end{aligned}
	$$
	In general, $\CM^{\op|L}$ and $\CM^{\op|R}$ are not equivalent as right $\CC$-modules (see \cite[Remark 7.1.5]{EGNO} or \cite[Section 2.1]{KZ1} for details).

	An algebra in a monoidal category $\CC$ is an object $A \in \CC$ equipped with two morphisms $u_A : \one_\CC \to A$ and $m_A : A \otimes A \to A$ in $\CC$ satisfying the unity and associativity properties (see \cite[Section 3.1]{O}).

	We have the following important proposition, which gives a reconstruction of a finite module category:
\begin{prop} \label{prop:rec-mod}
	Let $\CC$ be a multi-tensor category, and $_\CC\CM$ a finite left $\CC$-module. Then, we have $\CM \simeq \RMod_A(\CC)$ for some algebra $A \in \Alg(\CC)$.
\end{prop}
	For the proof of this proposition, we refer readers to \cite[Theorem 3.1]{O} or \cite[Proposition 2.3.9]{KZ1}.

%%%%%%%%%%%%%%%%%%%%%%%%%%%%%%%%%%%%%%%%%%%%%%%%%%%%%%%%

\subsection{Tensor product of module categories}	

	In this work, we focus on the tensor product of finite modules over a multi-tensor category. Similar notions for different cases can be seen in different texts, such as \cite{ENO1,EGNO,DSS,BBJ,Lu,KZ1,Gr}.
	
	Let $\CC$ be a multi-tensor category, $\CM_\CC$ a finite right $\CC$-module and $_\CC\CN$ a finite left $\CC$-module. A \textit{balanced $\CC$-module functor} is a right exact in each variable $k$-bilinear functor $F: \CM \times \CN \to \CD$ equipped with a natural isomorphism $F \circ (\odot \times \Id_\CN) \simeq F \circ (\Id_\CM \times \odot) : \CM \times \CC \times \CN \to \CD$ satisfying two evident relations. We use $\Fun_\CC^{bal}(\CM,\CN;\CD)$ to denote the category of balanced $\CC$-module functors $F: \CM \times \CN \to \CD$.
	
	The \textit{Deligne's tensor product} of $\CM_\CC$ and $_\CC\CN$ \textit{over $\CC$} is a $k$-linear abelian category $\CM \boxtimes_\CC \CN$, together with a balanced $\CC$-module functor $\boxtimes_\CC: \CM \times \CN \to \CM \boxtimes_\CC \CN$, such that, for every $k$-linear abelian category $\CD$, the composition with $\boxtimes_\CC$ induces an equivalence $\Fun(\CM \boxtimes_\CC \CN,\CD) \to \Fun^{bal}_\CC(\CM,\CN;\CD)$. In particular, in the special case $\CC=\bk$, the tensor product $\CM \boxtimes_\bk \CN$ is nothing but the ordinary Deligne's tensor product $\CM \boxtimes \CN$.
	
	We have the following important proposition to reconstruct $\CM \boxtimes_\CC \CN$:
\begin{prop} \label{prop:rec-tp}
	Let $\CC$ be a multi-tensor category. Let $\CM = \LMod_M(\CC)$, $\CM' = \RMod_M(\CC)$ and $\CN = \RMod_N(\CC)$ for some algebras $M,N$ in $\CC$. We have the following assertions:
	\begin{enumerate}
	\item The balanced $\CC$-module functor $\CM \times \CN \to \BMod_{M|N}(\CC)$ defined by $(x,y) \mapsto x \otimes y$ exhibits $\BMod_{M|N}(\CC)$ as the tensor product $\CM \boxtimes_\CC \CN$.
	\item The balanced $\CC$-module functor $\CM \times \CN \to \Fun_\CC(\CM',\CN)$ defined by $(x,y) \mapsto - \otimes_M x \otimes y$ exhibits $\Fun_\CC(\CM',\CN)$ as the tensor product $\CM \boxtimes_\CC \CN$.
	\end{enumerate}
\end{prop}
	For the proof of this proposition, we refer readers to \cite{KZ1}. Similar results for different cases appeared earlier in various texts, i.e. \cite{ENO3,DSS,BBJ,DN1,Lu}.
	
	We remind readers that $\CM \boxtimes_\CC \CN$ is also finite, for $\BMod_{M|N}(\CC)$ is finite. 

\begin{cor} \label{cor:obj-tp-quo}
	Every object $x$ in $\CM \boxtimes_\CC \CN$ is a quotient of some object with the form $m \boxtimes_\CC n$, where $m \in \CM, n \in \CN$. More precisely, $x$ is isomorphic to a coequalizer of $m' \boxtimes_\CC n' \rightrightarrows m \boxtimes_\CC n$ for some $m,m' \in \CM$ and $n,n' \in \CN$.
\end{cor}	
	
\begin{proof}
	Using Proposition \ref{prop:rec-mod}, identify $\CM$ with $\LMod_M(\CC)$ and $\CN$ with $\RMod_N(\CC)$ for some algebras $M,N$ in $\CC$. Then, via Theorem \ref{prop:rec-tp} identify $\CM \boxtimes_\CC \CN$ with $\BMod_{M|N}(\CC)$. Note that any object $x \in \BMod_{M|N}(\CC)$ can be seen as a left $M$-module by forgetting its right-module structure. Thus, we have $x \simeq x \otimes_N N$ is the coequalizer of $x \otimes N \otimes N \rightrightarrows x \otimes N$.
\end{proof}	
	
	The following remark gives a description of $\CM \boxtimes_\CC \CN$, which was first referred to in \cite{ENO3}:
\begin{rem} \label{rem:tp-m}
	By the assumption, the category $\CM \boxtimes \CN$ has a natural structure of a right $\CC \boxtimes \CC^\rev$-module. Let $R:\CC \to \CC \boxtimes \CC^\rev$ be the right adjoint functor of the tensor product functor $\otimes:\CC \boxtimes \CC^\rev \to \CC$. Set $L_\CC:=R(\one_\CC)$, which is exactly the canonical algebra in $\CC \boxtimes \CC^\rev$ and has a decomposition as $L_\CC=\oplus_{i \in \CO(\CC)}i^L \boxtimes i$, the summation taken over simple objects of $\CC$ (see \cite{EGNO} for details). Then, $\CM \boxtimes_\CC \CN$ is equivalent to $\RMod_{L_\CC}(\CM \boxtimes \CN)$, the category of	right $L_\CC$-modules in $\CM \boxtimes \CN$. 
	
	More precisely, the equivalence $\tilde{F}:\RMod_{L_\CC}(\CM \boxtimes \CN) \to \CM \boxtimes_\CC \CN$ can be given by the following commutative diagram:
	\begin{equation} \label{eq:tp-m}
	\xymatrix@C=48pt{
		\CM \boxtimes \CN \ar[r]^-{- \otimes L_\CC} \ar[d]|-{F:m \boxtimes n \mapsto m \boxtimes_\CC n} & \RMod_{L_\CC}(\CM \boxtimes \CN) \ar[ld]^-{\tilde{F}} \\
		\CM \boxtimes_\CC \CN. \\		
	}
	\end{equation} 
	
\end{rem}	
	
	Let $\CC$ and $\CD$ be two finite monoidal categories and $\CM$ a finite $\CC$-$\CD$-bimodule. We say that $\CM$ is \textit{invertible} if there is a finite $\CD$-$\CC$-bimodule $\CN$ such that there are equivalences $\CM \boxtimes_\CD \CN \simeq \CC$ as $\CC$-$\CC$-bimodules and $\CN \boxtimes_\CC \CM \simeq \CD$ as $\CD$-$\CD$-bimodules. If such an invertible bimodule exists, $\CC$ and $\CD$ are said to be \textit{Morita equivalent} (see \cite{Mu2,ENO2}).

%%%%%%%%%%%%%%%%%%%%%%%%%%%%%%%%%%%%%%%%%%%%%%%%%%%%%%%%

\subsection{Finite braided monoidal categories and monoidal modules}

	A \textit{finite braided monoidal category} over $k$ is a finite monoidal category $\CC$ endowed with a natural isomorphism for $x,y \in \CC$, $c_{x,y}: x \otimes y \to y \otimes x$ (named by \textit{braiding}), satisfying two relations (see \cite{JS2}). We say $\CC$ is \textit{symmetric} if $c_{x,y}=c_{y,x}^{-1}$.
	We use $\bar{\CC}$ to denote the same monoidal category as $\CC$ but endowed with the inverse braiding $\bar{c}_{x,y}=c_{y,x}^{-1}$. Note that given a symmetric monoidal category $\CE$, we have $\bar{\CE} \simeq \CE \simeq \CE^\rev$.

	Given a monoidal category $\CC$, recall that the \textit{Drinfeld center} of $\CC$, denoted by $\FZ_1(\CC)$ or $\FZ(\CC)$, is the category of pairs $(z,\gamma_{z,-})$, where $z \in \CC$ and $\gamma_{z,-}: z \otimes - \to - \otimes z$ is a half-braiding satisfying several relations (see \cite{JS1,Maj} for examples). Morphisms in $\FZ(\CC)$ are morphisms in $\CC$ between the first components respecting the half-braidings. 
	Note that the category $\FZ(\CC)$ has a natural structure of a braided monoidal category with the braiding defined by its half-braidings. Moreover, $\FZ(\CC)$ can be identified with the category of $\CC$-$\CC$-bimodule functors from $\CC$ to $\CC$. 
	In what follows, we use $\pi_\CC$ to denote the forgetful functor from $\FZ(\CC)$ to $\CC$.
	
	Let $\CC$ be a braided monoidal category with the braiding $c$. 
	We say two objects $x,y \in \CC$ \textit{centralize} each other if $c_{y,x} \circ c_{x,y} = \Id_{x \otimes y}$. For a braided monoidal full subcategory $\CD \subset \CC$, the \textit{centralizer} of $\CD$ in $\CC$, denoted by $\CD'|_\CC$, is defined to be the full subcategory of $\CC$ formed by those objects that centralize every object of $\CD$. It is easy to see $\CD'|_\CC$ is a braided monoidal category. In particular, $\CC'|_\CC$ is clearly a symmetric monoidal category, called the \textit{M\"uger center} of $\CC$ (see \cite{Mu3}). In this paper, we abbreviate $\CC'|_\CC$ by $\CC'$, and also denote it by $\FZ_2(\CC)$ (versus the notation $\FZ_1$ for the Drinfeld center). 
	
	In fact, we have a 2-category $\Cat_\bk$ of finite categories over $k$, right exact $k$-linear module functors and natural transformations. Furthermore, $\Cat_\bk$ equipped with the Deligne's tensor product and $\bk$ as the tensor unit form a symmetric monoidal 2-category. Then in the language of $E_n$-algebra (see \cite{Lu}), in $\Cat_\bk$, for a finite category $\CC$, $\Fun(\CC,\CC)$ is its $E_0$-center; for a finite monoidal category, the Drinfeld center is its $E_1$-center; and for a finite braided monoidal category, the M\"uger center is its $E_2$-center.
	
	Let $\CC$ and $\CD$ be finite braided monoidal categories. A \textit{monoidal $\CC$-$\CD$-bimodule} is a finite monoidal category $\CM$ equipped with a right exact $k$-linear braided monoidal functor $\psi_\CM: \bar{\CC} \boxtimes \CD \to \FZ(\CM)$ (see \cite{KZ1}). And naturally, a \textit{monoidal left $\CC$-module} is a monoidal $\CC$-$\bk$-bimodule, and a \textit{monoidal right $\CD$-module} is a monoidal $\bk$-$\CD$-bimodule. 
	
	We say a monoidal $\CC$-$\CD$-bimodule $\CM$ is \textit{closed} if $\psi_\CM$ is an equivalence.
	
	We say that two monoidal $\CC$-$\CD$-bimodules $\CM,\CN$ are \textit{equivalent} if there is a $k$-linear monoidal equivalence $F:\CM\to\CN$ such that the composite braided monoidal functor $\FZ(F) \circ \psi_\CM: \bar\CC\boxtimes\CD \xrightarrow{\psi_\CM} \FZ(\CM)\xrightarrow[\simeq]{\FZ(F)} \FZ(\CN)$ is isomorphic to $\psi_\CN$. 
	
\begin{exam} \label{exam:fun_CD}
	Let $\CC,\CD$ be multi-tensor categories and $_\CC\CM_\CD$ a finite $\CC$-$\CD$-bimodule. Then $\CW:=\Fun_{\CC|\CD}(\CM,\CM)$ is a monoidal $\FZ(\CC)$-$\FZ(\CD)$-bimodule. Without	loss of generality, we may assume $\CD=\bk$. The monoidal functor $\psi_\CW : \overline{\FZ(\CC)} \to \Fun_\CC(\CM,\CM)$ is given by $c \mapsto c \odot -$, where $\psi_\CW(c) = c \odot -$ is equipped with a natural isomorphism for $c' \in \CC$:
	$$S^{\psi_\CW(c)}_{c',-}:c' \odot \psi_\CW(c)(-) = c' \odot c \odot - \xrightarrow{\gamma_{c,c'}^{-1} \odot \Id_-} c \odot c' \odot - = \psi_\CW(c)(c' \odot -),$$
	which makes $\psi_\CW(c)$ a left $\CC$-module functor. Moreover, $\psi_\CW(c)$ is equipped with a natural half-braiding 
	$$\gamma_{\psi_\CW(c),F}: \psi_\CW(c) \circ F = c \odot F(-) \xrightarrow{S^F_{c,-}} F(c \odot -) = F \circ \psi_\CW(c),$$
	thus defines an object in $\FZ(\Fun_\CC(\CM,\CM))$. Also note that $\psi_\CW$ preserves the half-braiding, thus is promoted to a braided monoidal functor.
	
	In the special case $\CM = \CC$, the above construction recovers the canonical braided monoidal equivalence $\FZ(\CC) \simeq \FZ(\CC^\rev)$. 
\end{exam}
	
	We say that a braided fusion category $\CC$ is \textit{nondegenerate}, if the evident braided monoidal functor $\bar{\CC} \boxtimes \CC \to \FZ(\CC)$ is an equivalence, i.e. the monoidal bimodule $_\CC \CC_\CC$ is closed.
	There are several equivalent conditions of the nondegeneracy of a braided fusion category in \cite{DGNO}. One is that, for a braided fusion category $\CC$, its M\"uger center is trivial, i.e. $\CC' = \bk$.
	
	Let $\CB$ be a braided multi-tensor category, and let $\CC$ be a monoidal right $\CB$-module, $\CD$ a monoidal left $\CB$-module. Then the Deligne's tensor product $\CC \boxtimes_\CB \CD$ over $\CB$ has a canonical structure of a
	finite monoidal category with tensor unit $\one_\CC \boxtimes_\CB \one_\CD$ and tensor product $(x \boxtimes_\CB y) \otimes (x' \boxtimes_\CB y') = (x \otimes x') \boxtimes_\CB (y \otimes y')$ (see \cite{Gr}).

%%%%%%%%%%%%%%%%%%%%%%%%%%%%%%%%%%%%%%%%%%%%%%%%%%%%%%%%

\subsection{Functoriality of Drinfeld center}
	
	In this subsection, we assume that $k$ is an algebraically closed field.
	
	We recall two symmetric monoidal categories $\mtc_k^\ind$ and $\bmtc_k$ introduced in \cite{KZ1}:
	\begin{enumerate}
	\item The category $\mtc_k^\ind$ of indecomposable multi-tensor categories over $k$ with morphisms given by the equivalence classes of finite bimodules.
	\item The category $\bmtc_k$ of braided tensor categories over $k$ with morphisms given by the equivalence classes of monoidal bimodules. 
	\end{enumerate}
	Both categories are equipped with the Deligne's tensor product as the tensor product and $\bk$ as the tensor unit.
	
	Then, we can make Drinfeld center a functor as follows (see \cite{KZ1}):
	\begin{thm} \label{thm:functorial}
		The assignment $$\CC \mapsto \FZ(\CC), \ {}_\CC\CM_\CD \mapsto \FZ^{(1)}(\CM) := \Fun_{\CC|\CD}(\CM,\CM)$$ defines a symmetric monoidal functor $$\FZ: \mtc_k^\ind \to \bmtc_k.$$
	\end{thm}
	We refer to this functor $\FZ$ as the \textit{Drinfeld center functor}.
	
\begin{rem}
	Recall that, by Example \ref{exam:fun_CD}, $\Fun_{\CC|\CD}(\CM,\CM)$ has a structure of a monoidal $\FZ(\CC)$-$\FZ(\CD)$-bimodule.
	The functoriality of $\FZ$ (respecting the compositions of morphisms) follows from the equivalence
	\begin{equation}
	\begin{aligned}
	\Fun_{\CB|\CC}(\CM_1,\CN_1) \boxtimes_{\FZ(\CC)} \Fun_{\CC|\CD}(\CM_2,\CN_2) &\simeq \Fun_{\CB|\CD}(\CM_1\boxtimes_\CC\CM_2,\CN_1\boxtimes_\CC\CN_2), \\
	f \boxtimes_{\FZ(\CC)} g &\mapsto f \boxtimes_\CC g ,
	\end{aligned}
	\end{equation}
	and the fact that it is a monoidal equivalence when $\CN_1=\CM_1$,$\CN_2=\CM_2$ \cite[Theorem 3.1.7]{KZ1}.
\end{rem}
	
	In the rest of this subsection, we assume that $\chara k = 0$. 
	We recall two symmetric monoidal subcategories $\mfus_k^\ind \subset \mtc_k^\ind$ and $\bfus_k^\cl \subset \bmtc_k$ introduced in \cite{KZ1}.
	\begin{enumerate}
	\item The category $\mfus_k^\ind$ of indecomposable multi-fusion categories over $k$ with morphisms given by the equivalence classes of nonzero semisimple bimodules.
	\item The category $\bfus_k^\cl$ of nondegenerate braided fusion categories over $k$ with morphisms given by the equivalence classes of closed multi-fusion bimodules. 
	\end{enumerate}

	Now we are ready to state the main theorem of \cite{KZ1}.

\begin{thm} \label{thm:fully-faithful}
	The Drinfeld center functor $\FZ: \mtc^\ind_\bk \to \bmtc_\bk$ restricts to a fully faithful symmetric monoidal functor
	$$\FZ: \mfus^\ind_\bk \to \bfus^\cl_\bk. $$
\end{thm}

	There are two well-known results proved by Etingof, Nikshych and Ostrik in \cite{ENO2,ENO3}, which can be regarded as corollaries of Theorem \ref{thm:fully-faithful}: 

\begin{cor} \label{cor:Me-Ze} (\cite{ENO2})
	Two multi-fusion categories $\CC$ and $\CD$ are Morita equivalent if and only if $\FZ(\CC) \simeq^{br} \FZ(\CD)$, i.e. $\FZ(\CC)$ is braided monoidally equivalent to $\FZ(\CD)$.
\end{cor}

	Given a multi-fusion category $\CC$, we denote the group of the equivalence classes of semisimple invertible $\CC$-$\CC$-bimodules by $\mathrm{BrPic}(\CC)$, and denote the group of the isomorphism classes of braided auto-equivalences of $\FZ(\CC)$ by $\mathrm{Aut}^{br}(\FZ(\CC))$. 
	
\begin{cor} \label{cor:BP-Aut} (\cite{ENO3})
	Let $\CC$ be an indecomposable multi-fusion category.
	We have $\mathrm{BrPic}(\CC) \simeq \mathrm{Aut}^{br}(\FZ(\CC))$ as groups.
\end{cor}

%%%%%%%%%%%%%%%%%%%%%%%%%%%%%%%%%%%%%%%%%%%%%%%%%%%%%%%%

\section{Symmetry enriched categories} \label{sec:SEC}

	In this section, we give definitions for and focus on some kinds of categories with symmetry, and collectively call them \textit{symmetry enriched categories}. Let $k$ be a field and $\CE$ be a symmetric multi-tensor category over $k$ with the braiding $\beta$.

%%%%%%%%%%%%%%%%%%%%%%%%%%%%%%%%%%%%%%%%%%%%%%%%%%%%%%%%

\subsection{Some properties of a symmetric multi-tensor category and its modules}
	
	Because of the symmetry of $\CE$, a left $\CE$-module is also a right $\CE$-module, induced by the evident equivalence $\CE^\rev \simeq \CE$. 
	More precisely, let $\CC$ be a left $\CE$-module, whose left $\CE$-action is denoted by $\odot : \CE \times \CC \to \CC , (e , x) \mapsto e \odot x$ and	left-$\CE$-module associativity constraint is denoted by $m_{e',e,x}:(e' \otimes e) \odot x \to e' \odot (e \odot x)$.
	Then, there is a canonical right-$\CE$-module structure on it as follows: 
	
	The right $\CE$-action $\odot : \CC \times \CE \to \CC , (x , e) \mapsto x \odot e$ is exactly defined by its left $\CE$-action, i.e. $x \odot e = e \odot x$ (for convenience, we denote this equality by $b_{x,e}$).
	
	The right-$\CE$-module associativity constraint $m_{x,e,e'}$ is defined as the composition 
	$$
	\begin{aligned}
	x \odot (e \otimes e') 
	&\xrightarrow{b_{x, e \otimes e'}} (e \otimes e') \odot x \xrightarrow{\beta_{e,e'} \odot \Id_x} (e' \otimes e) \odot x \\	&\xrightarrow{m_{e',e,x}} e' \odot (e \odot x) 
	\xrightarrow{b^{-1}_{e \odot x,e'}} (e \odot x) \odot e' \xrightarrow{b^{-1}_{x,e} \odot e'} (x \odot e) \odot e'. \\
	\end{aligned}
	$$  
	(Similar analysis can be seen in \cite{Gr} for module categories over a braided monoidal category.)
	
	Thus, we can briefly collectively use the notion of an \textit{$\CE$-module} to substitute that of a left $\CE$-module or a right $\CE$-module. Meanwhile, a left-$\CE$-module functor is also a right-$\CE$-module functor, so we can briefly call it an \textit{$\CE$-module functor}. 
	
	%	(Similar analysis can be seen in \cite[Section 8.7]{EGNO}.) 
	
	Recall that, given a (left) $\CE$-module $\CM$, we have two ways to make $\CM^\op$ a (right) $\CE$-module ($\CM^{\op|L}$ and $\CM^{\op|R}$). Due to the symmetry of $\CE$, we are confident that these two ways are equivalent, i.e. there is an $\CE$-module equivalence $$\CM^{\op|R} \simeq \CM^{\op|L}.$$
	
	To give the equivalence precisely, we need the following notions and facts :
		
	Given a rigid monoidal category $\CC$, a \textit{pivotal structure} on $\CC$ is an isomorphism of monoidal functors $\alpha : \Id_\CC \to \delta^{LL}$, i.e. a natural isomorphism $\alpha_x : x \to x^{LL}, x \in \CC$ such that $\alpha_{x \otimes y} = \alpha_x \otimes \alpha_y$ for any $x,y \in \CC$.
	In addition, if $\CC$ is braided with the braiding $c$, we can define the \textit{Drinfeld morphism} $u_x$ (see \cite[Section 2.2]{BK} or \cite[Chapter XIV]{Ka}) as the composition $$x \xrightarrow{\Id_x \otimes \coev_{x^L}} x \otimes x^L \otimes x^{LL} \xrightarrow{c_{x,x^L} \otimes \Id_{x^{LL}}} x^L \otimes x \otimes x^{LL} \xrightarrow{\ev_x \otimes \Id_{x^{LL}}} x^{LL}.$$ According to \cite[Proposition 8.9.3]{EGNO}, we have that $u_x \otimes u_y = u_{x \otimes y} \circ c_{y,x} \circ c_{x,y}$ for $x,y \in \CC$.
	Furthermore, if $\CC$ is symmetric, $u$ is a pivotal structure.
	
	Thus, the Drinfeld morphism $u_{e^R} : e^R \to e^L , e \in \CE$ directly induces an $\CE$-module equivalence $\CM^{\op|R} \simeq \CM^{\op|L}$, and we can use the notation $\CM^\op$ to substitute $\CM^{\op|L}$ or $\CM^{\op|R}$.

	Let $A$ be an algebra in $\CE$. Then, we claim that any right $A$-module $x \in \CE$ automatically has a structure of a left $A$-module. Without loss of generality, we just consider free modules: for a free right $A$-module $e \otimes A$, one can define its left $A$-action as the composed morphism $$A \otimes e \otimes A \xrightarrow{\beta_{A,e} \otimes \Id_A} e \otimes A \otimes A \xrightarrow{\Id_e \otimes m_A} e \otimes A,$$ where $\beta$ denotes the braiding in $\CE$ (recall that $A$ is also an object in $\CE$) and $m_A$ is the multiplication of the algebra $A$.

	Next, we consider the decomposition of $\CE$.
	Assume $\one = \oplus_{i \in \Lambda} e_i$ with $e_i$ simple. Let $\CE_{ij} = e_i \otimes \CE \otimes e_j$. Recall that for $i \neq j$, we have $e_i \otimes e_j \simeq 0$. Then, we have $\CE_{ij} = e_i \otimes \CE \otimes e_j \simeq e_i \otimes e_j \otimes \CE \simeq 0$, which leads $\CE \simeq \oplus_{i \in \Lambda} \CE_{ii}$. (We refer readers to \cite{KZ1} for the theory of the decomposition of a multi-tensor category.)
	
	As a consequence, we have the following fact:
\begin{rem} \label{rem:E-ten-ind}
	For a symmetric multi-tensor category $\CE$, it is indecomposable (as a multi-tensor category) if and only if $\CE$ is a tensor category, i.e. $\one_\CE$ is simple.
\end{rem}

\subsection{Finite categories over $\CE$}

\begin{defn} \label{def:c_E}
	A \textit{finite category over $\CE$} is a finite $\CE$-module. We say a functor is \textit{$\CE$-linear} if it is an $\CE$-module functor between two finite categories over $\CE$.
\end{defn}

	Naturally, we have a 2-category $\Cat_\CE$ of finite categories over $\CE$, $\CE$-linear functors and natural transformations between $\CE$-linear functors.

	Since a finite category over $\CE$ has both left and right $\CE$-actions, the Deligne's tensor product over $\CE$, $\boxtimes_\CE$, is well-defined between two finite categories over $\CE$, and we have the following facts:
	
\begin{rem}
	Let $\CB$,$\CC$ and $\CD$ be finite categories over $\CE$. It follows directly from the definition that 
	\begin{enumerate}
	\item $\CC \boxtimes_\CE \CD$ is also a finite category over $\CE$. 
	\item the evident equivalence $\CC \boxtimes_\CE \CD \simeq \CD \boxtimes_\CE \CC$ and $(\CB \boxtimes_\CE \CC) \boxtimes_\CE \CD \simeq \CB \boxtimes_\CE (\CC \boxtimes_\CE \CD)$ are the equivalence of categories over $\CE$.
	\end{enumerate}
\end{rem}
	Thus, we obtain that $(\Cat_\CE,\boxtimes_\CE,\CE)$ defines a symmetric monoidal 2-category.

\begin{rem}
	For $\CX_1,\CX_2,...,\CX_n \in \Cat_\CE$, similar to the \textit{coherence theorem} in monoidal categories (see \cite{Mac} or \cite[Theorem 2.9.2]{EGNO} for details), on the product $\CX_1 \boxtimes_\CE ... \boxtimes_\CE \CX_n$, all the $\CE$-module structures (induced by $\CX_i$) are equivalence, and thus one can safely identify all of them with each other.
	% $\CE$-module action can be writen at any position and with any direction.
\end{rem}

	We have an extrinsic explanation on the monoidal structure of the 2-category $\Cat_\CE$ as follows. (Similar analysis for the 2-category formed by fusion categories over $\CE$ can be seen in \cite{DNO}.)
	
	Firstly, let us see the functoriality of 2-categories $\Cat_\CE$ with respect to symmetric monoidal functors $\Omega:\CE \to \CF$, for another symmetric multi-tensor category $\CF$. Construct a \textit{base change 2-functor} 
	$$- \boxtimes_\CE \CF : \Cat_\CE \to \Cat_\CF , \ \CC \mapsto \CC \boxtimes_\CE \CF,$$
	where the $\CE$-module structure on $\CF$ is given by the composed functors $\CE \times \CF \xrightarrow{\Omega \times \Id_\CF} \CF \times \CF \xrightarrow{\otimes} \CF$. Indeed, the base change 2-functors are monoidal:
	$$(\CC \boxtimes_\CE \CD) \boxtimes_\CE \CF \simeq (\CC \boxtimes_\CE \CF) \boxtimes_\CF (\CD \boxtimes_\CE \CF).$$
	
	Also note that the product induces a 2-functor
	$$\Cat_\CE \times \Cat_\CF \to \Cat_{\CE \times \CF} , \ (\CC,\CD) \mapsto \CC \times \CD.$$
	Assume that $\CF = \CE$ and compose the above 2-functor with the base change 2-functor 
	$$- \boxtimes_{\CE \times \CE} \CE : \Cat_{\CE \times \CE} \to \Cat_\CE$$
	induced by the tensor product functor $\CE \times \CE \xrightarrow{\boxtimes_\CE} \CE$.
	This brings a 2-functor 
	$$\Cat_\CE \times \Cat_\CE \to \Cat_\CE , \ (\CC,\CD) \mapsto \CC \boxtimes_\CE \CD,$$
	which gives the monoidal structure on the 2-category $\Cat_\CE$.

\subsection{Finite monoidal categories over $\CE$}	
	Definitions similar to the following two appeared earlier in \cite[Definition 4.16]{DGNO}.

\begin{defn} \label{def:mc_E}
	A \textit{finite monoidal category over $\CE$} is a finite monoidal category $\CC$ equipped with a right exact $k$-linear braided monoidal functor $\phi_\CC:\CE \to \FZ_1(\CC)$. 
	We say that $\CC$ is \textit{fully faithful over $\CE$} if $\phi_\CC$ is fully faithful and carries simple objects to simple objects.
\end{defn}

\begin{defn} \label{def:bmc_E}
	A \textit{finite braided monoidal category over $\CE$} is a finite braided monoidal category $\CC$ equipped with a right exact $k$-linear braided monoidal functor $\phi_\CC:\CE \to \FZ_2(\CC)$.
	We say that $\CC$ is \textit{fully faithful over $\CE$} if $\phi_\CC$ is fully faithful and carries simple objects to simple objects.
\end{defn}
	
\begin{rem}
	It follows directly from the definition that
	\begin{enumerate}
		\item If $\CC$ is a finite monoidal category over $\CE$, so is $\CC^\rev$.
		\item If $\CC$ is a finite braided monoidal category over $\CE$, so is $\bar{\CC}$.
	\end{enumerate}
\end{rem}
	
	One can also add rigidity and semisimplicity conditions to these notions. 
	For rigid cases, we need assume $k$ is an algebraically closed field and $\CE$ is a symmetric fusion category over $k$; and for semisimple cases, we need assume $\chara k= 0$ in addition. In what follows, when talking about these cases, we automatically add these assumptions.
	Then we have the following definitions:
\begin{defn} \label{def:tc_E}
	\begin{enumerate}
		\item A \textit{multi-tensor category over $\CE$} is a rigid finite monoidal category over $\CE$. A \textit{tensor category over $\CE$} is a multi-tensor category over $\CE$ with a simple tensor unit.
		\item A \textit{multi-fusion category over $\CE$} is a semisimple multi-tensor category over $\CE$. A \textit{fusion category over $\CE$} is a semisimple tensor category over $\CE$.
	\end{enumerate}
\end{defn}

\begin{defn} \label{def:btc_E}
	\begin{enumerate}
		\item A \textit{braided multi-tensor category over $\CE$} is a rigid finite braided monoidal category over $\CE$. A \textit{braided tensor category over $\CE$} is a braided multi-tensor category over $\CE$ with a simple tensor unit.
		\item A \textit{braided multi-fusion category over $\CE$} is a semisimple braided multi-tensor category over $\CE$. A \textit{braided fusion category over $\CE$} is a semisimple braided tensor category over $\CE$.
	\end{enumerate}
\end{defn}

\begin{defn} \label{def:center_E}
	Let $\CC$ be a finite monoidal category over $\CE$. The \textit{$/\CE$-center} $\FZ_{/\CE}(\CC)$ of $\CC$ is defined as the centralizer of $\CE$ in the Drinfeld center $\FZ_1(\CC)$, i.e. $\FZ_{/\CE}(\CC):=\CE'|_{\FZ_1(\CC)}$.
\end{defn}

\begin{rem}
	We can see Definition \ref{def:mc_E} \ref{def:bmc_E} and \ref{def:center_E} in the sense of $E_n$-algebras (see \cite{Lu}).
	In fact, in the symmetric monoidal 2-category $\Cat_\CE$, 
	\begin{enumerate}
	\item a finite monoidal category over $\CE$ is an $E_1$-algebra;
	\item a finite braided monoidal category over $\CE$ is an $E_2$-algebra.
	\end{enumerate}
	Moreover, we have
	\begin{enumerate}
	\item for a finite category $\CC$ over $\CE$, the $E_0$-center of $\CC$ is nothing but $\Fun_\CE(\CC,\CC)$, which is a finite monoidal category over $\CE$;
	\item for a finite monoidal category over $\CE$, the $/\CE$-center is exactly its $E_1$-center, which is a finite braided monoidal category over $\CE$;
	\item for a finite braided monoidal category over $\CE$, the M\"uger center is its $E_2$-center.
	\end{enumerate}
	As a consequence, for finite monoidal categories $\CC,\CD$ over $\CE$, $\CC \boxtimes_\CE \CD$ is also a finite monoidal category over $\CE$.
	From another point of view, we have a right exact $k$-linear braided monoidal functor $\CE \simeq \CE \boxtimes_\CE \CE \xrightarrow{\phi_\CC \boxtimes_\CE \phi_\CD} \FZ(\CC) \boxtimes_\CE \FZ(\CD) \to \FZ(\CC \boxtimes_\CE \CD)$.
\end{rem}
	
	Let us consider the notion of a \textit{finite left (or right) $\CC$-module category over $\CE$}, where $\CC$ is a finite monoidal category over $\CE$. A finite left $\CC$-module category over $\CE$ is a finite category over $\CE$ firstly. And it also admits a left $\CC$-module structure. Finally, these two structures must be compatible. Note that a finite left $\CC$-module automatically has a $\CE$-module structure induced by its left $\CC$-module structure: $\CE \xrightarrow{\phi_\CC} \FZ(\CC) \xrightarrow{\pi_\CC} \CC \xrightarrow{c \ \mapsto \ c \odot -} \Fun(\CM,\CM)$, and these two structures are compatible naturally. Thus, we do not need to consider the $\CE$-module structure independently. 

\begin{rem}
	Let $\CC$ be a multi-tensor category over $\CE$ and $_\CC\CM,_\CC\CN$ finite left $\CC$-modules. Then we have $\Fun_\CC(\CM,\CN)$ is a finite category over $\CE$ and $\Fun_\CC(\CM,\CM)$ (called the \textit{dual category} of $\CC$ in \cite{EO}) is a finite monoidal category over $\CE$.
\end{rem}	
		
	Then we give the definition of finite bimodules over $\CE$ as follows:	
\begin{defn} \label{def:bim_E}
	Let $\CC,\CD$ be finite monoidal categories over $\CE$. A \textit{finite $\CC$-$\CD$-bimodule over $\CE$} is a finite left $\CC \boxtimes_\CE \CD^\rev$-module. 
\end{defn}
	
\begin{rem}
	It is clear from definition that a finite left $\CC$-module over $\CE$ is exactly a finite $\CC$-$\CE$-bimodule over $\CE$; a finite right $\CC$-module over $\CE$ is exactly a finite $\CE$-$\CC$-bimodule over $\CE$; a finite category over $\CE$ is exactly a finite $\CE$-$\CE$-bimodule over $\CE$.
\end{rem}

\begin{rem} \label{rem:bim_E}
	In other words, a finite $\CC$-$\CD$-bimodule over $\CE$ is a finite $\CC$-$\CD$-bimodule $_\CC\CM_\CD$ equipped with a natural isomorphism between two $\CC$-$\CD$-bimodule functors $\eta_{\CM,e}:\pi_\CC(\phi_\CC(e)) \odot - \to - \odot \pi_\CD(\phi_\CD(e))$ for $e \in \CE$.
	In what follows, we omit the forgetful functors $\pi$ for simplicity unless confusion is possible.
\end{rem}

\begin{defn}
	Let $\CC$ and $\CD$ be two finite monoidal categories over $\CE$ and $_\CC\CM_\CD$ a finite $\CC$-$\CD$-bimodule over $\CE$. 
	We say that $\CM$ is \textit{invertible} if there is a finite $\CD$-$\CC$-bimodule $\CN$ over $\CE$ such that there are equivalences $\CM \boxtimes_\CD \CN \simeq \CC$ as $\CC$-$\CC$-bimodules over $\CE$ and $\CN \boxtimes_\CC \CM \simeq \CD$ as $\CD$-$\CD$-bimodules over $\CE$. 
	If such an invertible bimodule over $\CE$ exists, $\CC$ and $\CD$ are said to be \textit{Morita equivalent over $\CE$}.
\end{defn}

\subsection{Finite braided monoidal categories containing $\CE$ and monoidal modules} 

\begin{defn} \label{def:bmccE}
	A \textit{finite braided monoidal category containing $\CE$} is a finite braided monoidal category $\CC$ equipped with a right exact $k$-linear braided monoidal functor $\phi_\CC:\CE \to \CC$. 
	We say that $\CC$ is \textit{fully faithful containing $\CE$} if $\phi_\CC$ is fully faithful and carries simple objects to simple objects.
\end{defn}

\begin{rem}
	The following facts are clear from Definition \ref{def:bmccE}:
	\begin{enumerate}
	\item For a finite monoidal category $\CC$ over $\CE$, its Drinfeld center $\FZ(\CC)$ is a finite braided monoidal category containing $\CE$;
	\item For a finite braided monoidal category $\CC$ containing $\CE$, the centralizer of $\CE$ in $\CC$ is a finite braided monoidal category over $\CE$.
	\item A finite braided monoidal category $\CC$ over $\CE$ can be seen as a finite braided monoidal category containing $\CE$, with the braided monoidal functor $\CE \to \CC$ defined as the composition $\CE \to \FZ_2(\CC) \hookrightarrow \CC$.
	\item A finite braided monoidal category $\CC$ containing $\CE$ can be seen as a finite monoidal category over $\CE$, with the braided monoidal functor $\CE \to \FZ_1(\CC)$ defined as the composition $\CE \to \CC \hookrightarrow \FZ_1(\CC)$.	
	\end{enumerate}
\end{rem}

	We also have the following definitions for rigid or semisimple cases:
\begin{defn} \label{def:btccE}
	\begin{enumerate}
	\item A \textit{braided multi-tensor category containing $\CE$} is a rigid finite braided monoidal category containing $\CE$. A \textit{braided tensor category containing $\CE$} is a braided multi-tensor category containing $\CE$ with a simple tensor unit.
	\item A \textit{braided multi-fusion category containing $\CE$} is a semisimple braided multi-tensor category containing $\CE$. A \textit{braided fusion category containing $\CE$} is a semisimple braided tensor category containing $\CE$.
	\end{enumerate}
\end{defn}	
	
\begin{defn} \label{def:bmfE}
	Let $\CC,\CD$ be two braided monoidal categories containing $\CE$. A \textit{braided monoidal functor respecting $\CE$} is a braided monoidal functor $F:\CC \to \CD$ equipped with an isomorphism $\delta:F \circ \phi_\CC \simeq \phi_\CD$ of braided monoidal functors.
\end{defn}
	
\begin{prop} \label{prop:ZC-ZD-bmeE}
	Let $\CC,\CD$ be finite monoidal categories over $\CE$ and $_\CC\CM_\CD$ an invertible $\CC$-$\CD$-bimodule over $\CE$. Then, the evident monoidal functors $\FZ(\CC) \rightarrow \Fun_{\CC|\CD}(\CM,\CM) \leftarrow \FZ(\CD)$ are equivalences. Moreover, the induced equivalence $\FZ(\CC) \simeq \FZ(\CD)$ is a braided monoidal equivalence respecting $\CE$.
\end{prop}

	The usual case ($\CE=\bk$) of this proposition was proved in various texts with different assumptions (see for example \cite{S,Mu1,ENO3}). We sketch this part and then check the equivalence is indeed a braided monoidal equivalence respecting $\CE$.
	
\begin{proof}[Proof of Proposition \ref{prop:ZC-ZD-bmeE}]
	Let $_\CD \CN_\CC$ be an inverse of $\CM$. The evident functor $\CC \to \Fun_{\CD^\rev}(\CM,\CM)$ is a monoidal equivalence with its quasi-inverse given by the composed equivalence $\Fun_{\CD^\rev}(\CM,\CM) \simeq \Fun_{\CC^\rev}(\CM \boxtimes_\CD \CN,\CM \boxtimes_\CD \CN) \simeq \Fun_{\CC^\rev}(\CC,\CC) \simeq \CC $. It follows immediately that the evident monoidal functor (denoted by $L$) $$\FZ(\CC) \xrightarrow{L} \Fun_{\CC|\CD}(\CM,\CM), \ c \mapsto c \odot -$$ is an equivalence. Similarly, the evident monoidal functor (denoted by $R$) $$\FZ(\CD) \xrightarrow{R} \Fun_{\CC|\CD}(\CM,\CM), \ d \mapsto - \odot d$$ is an equivalence as well.
	
	Suppose the induced equivalence $\FZ(\CC) \simeq \FZ(\CD)$ carries $c,c'$ to $d,d'$, respectively. We have the following commutative diagram for $x \in \CM$:
	$$
	\xymatrix@C=48pt{
		c \odot (c' \odot x) \ar[r]^{\sim} \ar[d]_{\beta_{c,c'} \odot \Id_x} & (c' \odot x) \odot d \ar[r]^{\sim} \ar[d]^{\sim} & (x \odot d') \odot d \ar[d]^{\Id_x \odot \beta_{d',d} } \\
		c' \odot (c \odot x) \ar[r]^{\sim} & c' \odot (x \odot d) \ar[r]^{\sim} & (x \odot d) \odot d' \\
	}
	$$
	The commutativity of the left square is due to the fact that the isomorphism $c \odot - \simeq - \odot d$ is of left $\CC$-module functors; 
	that of the right square is due to the fact that $c' \odot - \simeq - \odot d'$ is of right $\CD$-module functors. Then the commutativity of the outer square says that the induced equivalence $\FZ(\CC) \simeq \FZ(\CD)$ preserves braidings.
	
	Recall that we have a natural isomorphism $\eta_\CM$ (see Remark \ref{rem:bim_E}) between the composed functors
	$$\CE \xrightarrow{\phi_\CC} \FZ(\CC) \xrightarrow{L} \Fun_{\CC|\CD}(\CM,\CM), \ e \mapsto \phi_\CC(e) \odot -,$$
	and 
	$$\CE \xrightarrow{\phi_\CD} \FZ(\CD) \xrightarrow{R} \Fun_{\CC|\CD}(\CM,\CM), \ e \mapsto - \odot \phi_\CD(e).$$
	This induced a natural isomorphism $\delta:(R^{-1} \circ L) \circ \phi_\CC \simeq \phi_\CD$. And $\delta$ preserves the monoidal structure and braidings because $\eta_\CM$ is a natural isomorphism between two $\CC$-$\CD$-bimodule functors.

\end{proof}	

\begin{defn} \label{def:mbimcE}
	Let $\CC,\CD$ be finite braided monoidal categories containing $\CE$. A \textit{monoidal $\CC$-$\CD$-bimodule containing $\CE$} is a finite monoidal category $\CM$ equipped with a right exact $k$-linear braided monoidal functor $$\psi_\CM:\bar\CC\boxtimes\CD \to \FZ(\CM)$$ 
	as well as an isomorphism $\eta_\CM$ between the monoidal functors 
	$$\psi_\CM^-: \CE \xrightarrow{\overline{\phi_\CC}\boxtimes\one_\CD} \bar\CC\boxtimes\CD \xrightarrow{\psi_\CM}\FZ(\CM) \to \CM,$$
	and
	$$\psi_\CM^+: \CE \xrightarrow{\one_{\bar\CC}\boxtimes\phi_\CD} \bar\CC\boxtimes\CD \xrightarrow{\psi_\CM}\FZ(\CM) \to \CM.$$

	We say that $\CM$ is \textit{closed} if $\psi_\CM$ is an equivalence.

	We say that two monoidal $\CC$-$\CD$-bimodules $_\CC\CM_\CD,_\CC\CN_\CD$ containing $\CE$ are  \textit{equivalent} if there is a $k$-linear monoidal equivalence $F:\CM\to\CN$ as well as an isomorphism $\theta$ between braided monoidal functors $\FZ(F) \circ \psi_\CM: \bar\CC\boxtimes\CD \xrightarrow{\psi_\CM} \FZ(\CM)\xrightarrow[\simeq]{\FZ(F)} \FZ(\CN)$ and $\psi_\CN$ such that the following diagram commutes for $e\in\CE$:
	$$
	\xymatrix@C=48pt{
		F(\psi_\CM^-(e)) \ar[r]^-{F(\eta_{\CM,e})} \ar[d]_\theta & F(\psi_\CM^+(e)) \ar[d]^\theta \\
		\psi_\CN^-(e) \ar[r]^-{\eta_{\CN,e}} & \psi_\CN^+(e). \\
	}
	$$
\end{defn}

\begin{rem} 
	Let $\CB,\CC,\CD$ be finite braided monoidal categories containing $\CE$, and let $_\CB \CM_\CC, _\CC \CN_\CD$ be monoidal bimodules containing $\CE$. Then $\CM \boxtimes_\CC \CN$ is a monoidal $\CB$-$\CD$-bimodule containing $\CE$ as well.
\end{rem}

	We give some examples of monoidal bimodules containing $\CE$:
\begin{exam} \label{exam:ZC-ZD-bimcE}
	Let $\CC,\CD$ be finite monoidal categories over $\CE$ and $_\CC\CM_\CD$ an invertible $\CC$-$\CD$-bimodule over $\CE$. By Proposition \ref{prop:ZC-ZD-bmeE}, we have  monoidal equivalences $\FZ(\CC) \xrightarrow{L} \Fun_{\CC|\CD}(\CM,\CM) \xleftarrow{R} \FZ(\CD)$ and a braided monoidal equivalence $R^{-1} \circ L : \FZ(\CC) \to \FZ(\CD)$ respecting $\CE$. In this case, we have 
\begin{enumerate}
	\item the monoidal functors $\FZ(\CC) \xrightarrow{L} \Fun_{\CC|\CD}(\CM,\CM) \xleftarrow{R} \FZ(\CD)$ define a structure of a monoidal $\FZ(\CC)$-$\FZ(\CD)$-bimodule containing $\CE$ on $\Fun_{\CC|\CD}(\CM,\CM)$;
	\item the monoidal functors $\FZ(\CC) \xrightarrow{R^{-1} \circ L} \FZ(\CD) \xleftarrow{\Id} \FZ(\CD)$ define a structure of a monoidal $\FZ(\CC)$-$\FZ(\CD)$-bimodule containing $\CE$ on $\FZ(\CD)$.
\end{enumerate}
	Then, the monoidal equivalence $R: \FZ(\CD) \simeq \Fun_{\CC|\CD}(\CM,\CM)$ defines a monoidal $\FZ(\CC)$-$\FZ(\CD)$-bimodule equivalence containing $\CE$.
\end{exam}

	Actually, even if $\CM$ is not invertible, there is a monoidal $\FZ(\CC)$-$\FZ(\CD)$-bimodule containing $\CE$ structure on $\Fun_{\CC|\CD}(\CM,\CM)$:
\begin{exam} \label{exam:fun_CDcE}
	Let $\CC,\CD$ be finite monoidal categories over $\CE$, and $_\CC\CM_\CD$ a finite $\CC$-$\CD$-bimodule over $\CE$. Then $\CW:=\Fun_{\CC|\CD}(\CM,\CM)$ is a monoidal $\FZ(\CC)$-$\FZ(\CD)$-bimodule containing $\CE$. More precisely, we have a braided monoidal functor 
	$$
	\psi_\CW: \overline{\FZ(\CC)} \boxtimes \FZ(\CD) \simeq \overline{\FZ(\CC)}\boxtimes\overline{\FZ(\CD^\rev)} \to \FZ(\CW)
	$$
	such that $\psi_\CW(c\boxtimes d) = c \odot - \odot d$ as objects of $\CW$, equipped with the evident half-braiding $c \odot F(-) \odot d \simeq F(c \odot - \odot d)$, where $c \in \overline{\FZ(\CC)}$, $d \in \overline{\FZ(\CD^\rev)}$, $F\in\CW$ (the same as \ref{exam:fun_CD}). The natural isomorphism $\eta_{\CW,e}:\phi_\CC(e) \odot - \odot \one_\CD \to \one_\CC \odot - \odot \phi_\CD(e)$ is induced by $\phi_\CC(e) \boxtimes_\CE \one_\CD \simeq \one_\CC \boxtimes_\CE \phi_\CD(e)$ for $e \in \CE$. 
\end{exam}	
	
\begin{defn}
	Assume $k$ is an algebraically closed field and $\CE$ is a symmetric fusion category over $k$. Let $\CC,\CD$ be braided multi-fusion categories containing $\CE$ and $\CM$ a monoidal $\CC$-$\CD$-bimodule containing $\CE$. We say that $\CM$ is a \textit{multi-fusion $\CC$-$\CD$-bimodule containing $\CE$} if $\CM$ is also a multi-fusion category.
\end{defn}

\subsection{$\CE$-module braidings and braided $\CE$-modules}
	In this subsection, inspired by a sort of common but vital observables---full braidings in physics, we fix our eyes on the following structure on $\CE$-modules:
\begin{defn} \label{def:br}
	Let $\CC$ be an $\CE$-module. An \textit{$\CE$-module braiding} on $\CC$ is a natural isomorphism $\tau_{e,x}: e\odot x \to e\odot x$ for $e \in \CE$, $x \in \CC$ such that $\tau_{\one,x} = \Id_{\one \odot x}$ and the following diagram
	\begin{equation} \label{diag:E-br}
	\xymatrix@C=48pt{
		e \odot e' \odot x \ar[r]^{\Id_e \odot \tau_{e',x}} \ar[rd]_{\tau_{e \otimes e',x}} & e \odot e' \odot x \ar[r]^{\beta_{e,e'} \odot \Id_x} \ar[d]^{\tau_{e,e' \odot x}} & e' \odot e \odot x \ar[d]^{\Id_{e'} \odot \tau_{e,x}} \\
		& e \odot e' \odot x \ar[r]^{\beta_{e,e'} \odot \Id_x} & e' \odot e \odot x. \\
	}
	\end{equation} 
	commutes for all $e,e' \in \CE$ and $x \in \CC$. 
	A \textit{braided} $\CE$-module is an $\CE$-module with an $\CE$-module braiding.
	
	We say an $\CE$-module braiding $\tau$ is \textit{trivial}, if $\tau_{e,x} = \Id_{e \odot x}$ for all $e \in \CE$ and $x \in \CC$.  
\end{defn}
	
	In \cite[Section 5.1]{Bro}, a notion also named braided module categories, equips a structure similar to Definition \ref{def:br} but over an ordinary  braided monoidal category. And the theory of braided module categories are further developed in \cite{DN2,J-FR}.

	Since this structure is vital throughout our work, we would like to give several examples:

\begin{exam} \label{exam:fun_CD-br1}
	Let $\CC,\CD$ be finite monoidal categories over $\CE$, and $_\CC\CM_\CD,_\CC\CN_\CD$ finite $\CC$-$\CD$-bimodules over $\CE$. Then $\CW := \Fun_{\CC|\CD}(\CM,\CN)$ is a braided $\CE$-module, where we choose the $\CE$-action as the one induced by the left $\CC$-action: 
	$$\odot^\CE: \CE \times \CW \to \CW, \ (e,f) \mapsto \phi_\CC(e) \odot f,$$ 
	and the $\CE$-module braiding is given as follows:
	$$
	\begin{aligned}
	\tau_{e,f}:
	\phi_\CC(e) \odot f(-)
	&\xrightarrow{\eta_{\CN,e} \odot \Id_f} f(-) \odot \phi_\CD(e)
	\xrightarrow{S_{\phi_\CD(e),-}^{f|R}} f(- \odot \phi_\CD(e)) \\
	&\xrightarrow{\Id_f \odot \eta_{\CM,e}^{-1}} f(\phi_\CC(e) \odot -) 
	\xrightarrow{(S_{\phi_\CC(e),-}^{f|L})^{-1}} \phi_\CC(e) \odot f(-), \\
	\end{aligned}
	$$
	for $e \in \CE$ and $(f,S_{-,-}^{f|L},S_{-,-}^{f|R}) \in \Fun_{\CC|\CD}(\CM,\CN)$. 
	Here, $\eta_\CM,\eta_\CN$ are the natural transformations equipped by finite $\CC$-$\CD$-bimodules $\CM,\CN$ over $\CE$ (see Remark \ref{rem:bim_E}). 
	
\end{exam}

\begin{exam} \label{exam:bmccE_br}
	Let $\CC$ be a braided monoidal category containing $\CE$. Then $\CC$ is a braided $\CE$-module, where the $\CE$-action is defined as $e \odot x := \phi_\CC(e) \otimes x$ and the $\CE$-module braiding $\tau_{e,x}$ is chosen as the double braiding between $e$ and $x$ in $\CC$, i.e. $\tau_{e,x}:=c_{x,\phi_\CC(e)} \circ c_{\phi_\CC(e),x}$. 
	In particular, a braided monoidal category over $\CE$ has a trivial $\CE$-module braiding.
\end{exam}
	
	we provide a view to see why the double braiding in \ref{exam:bmccE_br} satisfies the commutative diagram \ref{diag:E-br}:
\begin{enumerate}
	\item The commutativity of the left triangle means that the isomorphism that two object $e,e' \in \CE$ loop around $x \in \CC$ in turn is equivalent to the one that $e$ and $e'$ loop around $x$ together;
	\item The commutativity of the right square means that the isomorphism that $e$ loops around $e'$ and $x$ is equivalent to the one that $e$ loops around $x$ only (note that the double braiding between $e$ and $e'$ is trivial).
\end{enumerate}
	This is also the reason why we named the structure in Definition \ref{def:br} by an $\CE$-module braiding.

\begin{exam} \label{exam:mbimcE_br}
	Let $\CC,\CD$ be finite braided monoidal categories containing $\CE$, and $\CM$ a monoidal $\CC$-$\CD$-bimodule containing $\CE$. Then, $\CM$ is a finite braided $\CE$-module, where we choose the $\CE$-action as 
	$$\odot^\CE: \CE \times \CM \to \CM, \ (e,x) \mapsto \psi_\CM^-(e) \otimes x,$$ 
	and the $\CE$-module braiding is
	$$
	\begin{aligned}
	\tau_{e,x}: \psi_\CM^-(e)\otimes x
	&\xrightarrow{\eta_{\CM,e}\otimes\Id_x} \psi_\CM^+(e)\otimes x
	\xrightarrow{c_{\psi_\CM^+(e),x}} x\otimes\psi_\CM^+(e) \\
	&\xrightarrow{\Id_x\otimes\eta_{\CM,e}^{-1}} x\otimes\psi_\CM^-(e) 
	\xrightarrow{(c_{\psi_\CM^-(e),x})^{-1}} \psi_\CM^-(e)\otimes x, \\
	\end{aligned}
	$$
	where $c$ is the half-braiding in $\FZ(\CM)$.	
	
	Note that we can see $\CC$ as a monoidal $\CC$-$\CC$-bimodule containing $\CE$, where $\eta_\CC=\Id$.
	Then, the braiding $\tau_{e,x}$ we introduce here is exactly the double braiding between $e$ and $x$, which coincides with Example \ref{exam:bmccE_br}. That is why we choose this isomorphism (rather than its inverse) as the $\CE$-module braiding on $\CM$.
\end{exam}

\begin{exam} \label{exam:fun_CD-br2}
	Let $\CC,\CD$ be finite monoidal categories over $\CE$, and $_\CC\CM_\CD$ a finite $\CC$-$\CD$-bimodule over $\CE$. We have $\CW:=\Fun_{\CC|\CD}(\CM,\CM)$ is a monoidal $\FZ(\CC)$-$\FZ(\CD)$-bimodule containing $\CE$ (see Example \ref{exam:fun_CDcE}). Then by Example \ref{exam:mbimcE_br}, there is an $\CE$-module braiding on $\CW$.
	Note that this $\CE$-module braiding coincides with that mentioned in Example \ref{exam:fun_CD-br1} when $\CN =\CM$.
\end{exam}

\begin{exam} \label{exam:ModA-br}
	Let $\CC$ be a finite braided $\CE$-module with an $\CE$-module braiding $\tau$. Assume that $\CC=\RMod_A(\CE)$ where $A$ is an algebra in $\CE$. We claim that $\tau$ is determined by $\tau_{e,A}^\CC$ for all $e \in \CE$.
	
	For an object $x \in \RMod_A(\CE)$, i.e. a right $A$-module, $x$ must be a quotient of some free right $A$-module $e' \otimes A$ (actually, $x \simeq x \otimes_A A$ is a quotient of $x \otimes A$). 
	By the naturality of $\tau$, we know that the $\CE$-module braiding on $\CC$ is determined by those braidings between objects in $\CE$ and free modules in $\CC$. Moreover, these braidings are determined by $\{\tau_{e,A}\}_{e \in \CE}$.
	In fact, we have the following commutative diagram for $e \in \CE$:
	$$
	\xymatrix@C=48pt{
		e' \otimes e \otimes A \ar[r]^-{\beta_{e',e} \otimes \Id_A} \ar[d]^{\Id_{e'} \otimes \tau_{e,A}} & e \otimes e' \otimes A \ar[d]^{\tau_{e, e' \otimes A}} \ar@{->>}[r] & e \otimes x \ar[d]^{\tau_{e,x}} \\
		e' \otimes e \otimes A \ar[r]^-{\beta_{e',e} \otimes \Id_A} & e \otimes e' \otimes A \ar@{->>}[r] & e \otimes x . \\
	}
	$$
	Thus, we only need to study on $\tau^\CC_{e,A}$.
	
	Also recall that $x$ has a structure of a left $A$-module as well. 
	For a free right $A$-module $e \otimes A$, we can define its left $A$-action as the composed morphism $A \otimes e \otimes A \xrightarrow{\beta_{A,e} \otimes \Id_A} e \otimes A \otimes A \xrightarrow{\Id_e \otimes m_A} e \otimes A$, where $\beta$ denotes the half-braiding in $\CE$ (recall that $A$ is also an object in $\CE$) and $m_A$ is the multiplication of the algebra $A$.
	Thus, we can safely use the formula
	$$\tau^\CC_{e,x}=\tau^\CC_{e,A} \otimes_A \Id_x, \quad e \in \CE, x \in \CC.$$

\end{exam}

\begin{exam} \label{exam:Cop-br}
	Let $\CC$ be a finite braided $\CE$-module. Then, on $\CC^\op$ should be an $\CE$-module braiding as well. How to define it? 
	
	Suppose that $\CC=\RMod_A(\CE)$ where $A$ is an algebra in $\CE$. 
	
	Consider the definition of an $\CE$-module braiding, and we have the following commutative diagram for all $e \in \CE$:
	$$
	\xymatrix@C=48pt{
		A \otimes e \otimes A \ar[r]^-{\beta_{A,e} \otimes \Id_A} \ar[d]^{\Id_A \otimes \tau_{e,A}} & e \otimes A \otimes A \ar[d]^{\tau_{e,A \otimes A}} \ar[r]^-{\Id_e \otimes m_A} & e \otimes A \ar[d]^{\tau_{e,A}} \\
		A \otimes e \otimes A \ar[r]^-{\beta_{A,e} \otimes \Id_A} & e \otimes A \otimes A \ar[r]^-{\Id_e \otimes m_A} & e \otimes A . \\
	}
	$$
	The commutativity of the outer square shows that the right-$A$-module isomorphism $\tau_{e,A}: e \otimes A \to e \otimes A$ is also a left-$A$-module isomorphism, which leads $\tau_{e, x}$ is an $\CE$-module braiding on $\LMod_A(\CE)$ as well.
	Then, we obtain an $\CE$-module braiding on $\LMod_A(\CE)$:
	$$\tau^{\LMod_A(\CE)}_{e,x} = \tau^{\CC}_{e,x}, \quad e \in \CE, x \in \CC.$$
	
	Now, identify $\CC^{\op|L}$ with $\LMod_A(\CE)$ via the $\CE$-module equivalence $\delta^L: \LMod_A(\CE) \simeq \RMod_A(\CE)^{\op|L}$ (recall that the $\CE$-action on $\CC^{\op|L}$: $e \odot^L x := x \otimes e^L$), which induces the $\CE$-module braiding on $\CC^{\op|L}$:
	$$\tau^{\CC^\op}_{e,x}=(\tau^{\LMod_A(\CE)}_{e,x^R})^L=(\tau^\CC_{e,x^R})^L, \quad e \in \CE, x \in \CC.$$ 
	
	We remind readers that we take the right dual of $x$ via the rigidity of $\CE$ and this process involves a choice of the algebra $A$. However, by Corollary \ref{cor:op-br} and Remark \ref{rem:preseve-br}, we know that on $\CC^\op$ the $\CE$-module braiding defined in this way is indeed independent of the choice of $A$.
\end{exam}

\begin{defn} \label{def:Fun'}	
	Given braided $\CE$-modules $\CC,\CD$ with $\CE$-module braidings $\tau^\CC$ and $\tau^\CD$ respectively, we say that an $\CE$-module functor $F:\CC\to\CD$ \textit{preserves the $\CE$-module braiding} or is \textit{a braided $\CE$-module functor} if the following diagram
	\begin{equation} \label{fun'} 
	\xymatrix@C=48pt{
		F(e\odot x) \ar[r]^{S^F_{e,x}} \ar[d]_{F(\tau^\CC_{e,x})} & e\odot F(x) \ar[d]^{\tau^\CD_{e,F(x)}} \\
		F(e\odot x) \ar[r]^{S^F_{e,x}} & e\odot F(x) \\
	}		
	\end{equation}
	commutes for $e\in\CE$, $x\in\CC$.
	We use $\Fun_\CE'(\CC,\CD)$ to denote the full subcategory of $\Fun_\CE(\CC,\CD)$ formed by the braided $\CE$-module functors.
\end{defn}

	Just like the analysis in \ref{exam:ModA-br}, when assuming $\CC=\RMod_A(\CE)$ for an algebra $A$ in $\CE$, we only need to check if the diagram \eqref{fun'} commutes for $x=A$ to judge if $F \in \Fun'(\CC,\CD)$.

\begin{rem}
	Let $\CC,\CD$ be finite braided $\CE$-modules with $\CE$-module braidings $\tau^\CC$ and $\tau^\CD$ respectively. Then the tensor product bifunctor $\boxtimes_\CE$ carries these two $\CE$-module braidings to two ones on $\CC \boxtimes_\CE \CD$. Since it brings no confusion, we denote them by $\tau^\CC$ and $\tau^\CD$ as well. 
\end{rem}

\subsection{The relative tensor product of braided $\CE$-modules}
	After this preparatory work, we introduce the notion of the relative tensor product over $\CE$:

\begin{defn} \label{def:rtp}
	Let $\CC,\CD$ be finite braided $\CE$-modules with $\CE$-module braidings $\tau^\CC$ and $\tau^\CD$ respectively. The \textit{relative tensor product} $\CC \boxdot_\CE \CD$ \textit{over $\CE$} is defined to be the full subcategory of $\CC \boxtimes_\CE \CD$ formed by the objects $x$ such that $\tau^\CC_{e,x} = \tau^\CD_{e,x}$ for all $e \in \CE$.
\end{defn}

	If $\CC,\CD$ are finite $\CE$-modules with trivial $\CE$-module braidings (for example, finite braided categories over $\CE$), then $\CC \boxdot_\CE \CD = \CC \boxtimes_\CE \CD$. Another example is the special case $\CE=\bk$, in which the relative tensor product is nothing but the Deligne's tensor product.
	
	Next, we study some basic properties of the relative tensor product over $\CE$.
	
\begin{rem}
	It follows from the definition of $\CE$-module braidings that, as the full subcategory of $\CC \boxtimes_\CE \CD$, $\CC \boxdot_\CE \CD$ is stable under the $\CE$-action, which means it is a finite braided $\CE$-module. Moreover, $\CC \boxdot_\CE \CD$ is closed under taking subquotients or direct sums. 
\end{rem}

\begin{lem} \label{lem:max-quotient}
	Let $\CC,\CD$ be finite braided $\CE$-modules. Any object in $\CC \boxtimes_\CE \CD$ has a unique maximal quotient in $\CC \boxdot_\CE \CD$ up to an isomorphism, where the maximality means that the epimorphism factor through no other quotient in $\CC \boxdot_\CE \CD$.
\end{lem}
	
\begin{proof}
	Assume that $x\in \CC \boxtimes_\CE \CD$ and that $y,y' \in \CC \boxdot_\CE \CD$ are both maximal quotient objects of $x$. Then we consider the pushout $y'' = y \amalg_x y'$ and the pullback $y''' = y \times_{y''} y'$. 
	Note that both epimorphisms $x \twoheadrightarrow y$ and $x \twoheadrightarrow y'$ factor through $y'''$.
%	(In some sense, $y''$ is the intersection of $y$ and $y'$, and $y'''$ is the union of $y$ and $y'$.) 
%	the sense of the partially ordered set
		
	Next, we show that $y'''$ is also a quotient of $x$ in $\CC \boxdot_\CE \CD$.
	By the universal property, the map $x \to y'''$ is also an epimorphism. Also note that $y''' \hookrightarrow y \oplus y'$, which implies $y''' \in \CC \boxdot_\CE \CD$. 
	
	By the maximality of $y$ and $y'$, we get $y \simeq y''' \simeq y'$, which leads to the conclusion. 	
\end{proof}

\begin{prop} \label{prop:rtp-finite}
	Let $\CC,\CD$ be finite braided $\CE$-modules. Then $\CC\boxdot_\CE\CD$ is finite.
\end{prop}

\begin{proof}
	Consider the intrinsic description of a finite category. The only condition not obvious is that every simple object of $\CC \boxdot_\CE \CD$ has a projective cover. 
	
	For any simple object $x$ of $\CC \boxdot_\CE \CD$, let $p$ be the projective cover of $x$ in $\CC \boxtimes_\CE \CD$. Then, by Lemma \ref{lem:max-quotient}, $p$ has a unique maximal quotient in $\CC \boxdot_\CE \CD$, denoted by $q$. We claim that $q$ is the projective cover of $x$ in $\CC \boxdot_\CE \CD$.
	
	For any $y \in \CC \boxdot_\CE \CD$ and $f \in \Hom_{\CC \boxtimes_\CE \CD}(p,y)$, $f$ has the canonical decomposition $p \stackrel{f_0}{\twoheadrightarrow} \Img(f) \hookrightarrow y$. This implies $\Img(f) \in \CC \boxdot_\CE \CD$. Then, $f_0$ must uniquely factor through $q$. This leads to $\Hom_{\CC \boxtimes_\CE \CD}(p,i(-)) \simeq \Hom_{\CC \boxdot_\CE \CD}(q,-)$, where $i:\CC \boxdot_\CE \CD \hookrightarrow \CC \boxtimes_\CE \CD$ is the inclusion functor. Since $p$ is projective in $\CC \boxtimes_\CE \CD$, i.e. the functor $\Hom_{\CC \boxtimes_\CE \CD}(p,-)$ is exact, $q$ is projective in $\CC \boxdot_\CE \CD$ as well.
	Note that $q$ has only one simple quotient, $x$. Thus, $q$ is the projective cover of $x$ in $\CC \boxdot_\CE \CD$, which completes the proof.
\end{proof}

\begin{rem}
	If $\CC,\CD$ are finite braided monoidal categories containing $\CE$, then $\CC \boxdot_\CE \CD$ is a finite braided monoidal category containing $\CE$ as well.
\end{rem}

	One may notice that our definition of the relative tensor product over $\CE$ is by a construction. Indeed it satisfies a universal property as well.

\begin{defn}
	Let $\CC,\CD,\CW$ be finite braided $\CE$-modules.
	A \textit{braiding-preserved $\CE$-bilinear bifunctor} is a $\CE$-bilinear bifunctor $F: \CC \times \CD \to \CW$ which is right exact and preserves the $\CE$-module braidings separately in each variable, i.e. the following diagram commutes for $e \in \CE, c \in \CC, d \in \CD$:
	$$
	\xymatrix@C=48pt{
		F(e \odot c,d) \ar[r]^\sim \ar[d]_{F(\tau^\CC_{e,c},d)} & e \odot F(c,d) \ar[d]^{\tau^\CW_{e,F(c,d)}} & F(c,e \odot d) \ar[l]_\sim \ar[d]^{F(c,\tau^\CD_{e,d})}\\
		F(e \odot c,d) \ar[r]^\sim & e \odot F(c,d) & F(c,e \odot d). \ar[l]_\sim  \\
	}
	$$
	
	We use $\Fun_\CE^{br}(\CC,\CD;\CW)$ to denote the category of braiding-preserved $\CE$-bilinear bifunctors $F: \CC \times \CD \to \CW$.
\end{defn}

\begin{rem}
	We construct a functor $\boxdot_\CE$ from $\CC \times \CD$ to $\CC \boxdot_\CE \CD$.
	In $\CC \boxdot_\CE \CD$, pick the maximal quotient of $c \boxtimes_\CE d \in \CC \boxtimes_\CE \CD$ as the image of $(c,d)$. One can check that $\boxdot_\CE$ is a well-defined braiding-preserved $\CE$-bilinear bifunctor. Moreover, composition with $\boxdot_\CE$ induces an equivalence $\Fun'_\CE(\CC \boxdot_\CE \CD,\CW) \simeq \Fun^{br}_\CE(\CC,\CD;\CW)$, and the universal property of the relative tensor product is illustrated in the following commutative diagram
	$$
	\xymatrix@C=48pt{
		\CC \times \CD \ar[r]^{\boxdot_\CE} \ar[dr]_{F} & \CC \boxdot_\CE \CD \ar[d]_{\exists! \underline{F}}|{\in}^{\Fun'_\CE(\CC \boxdot_\CE \CD,\CW)} \\
		& \CW  \\
	}
	$$
	for $F \in \Fun^{br}_\CE(\CC,\CD;\CW)$.
\end{rem}	

\begin{rem}
	When $\CC,\CD$ are finite braided monoidal categories containing $\CE$, the functor $\boxdot_\CE$ is still braided but not monoidal in general. (It is oplax monoidal instead.)
\end{rem}

\begin{rem} \label{rem:commutation-associative}
	Let $\CB,\CC,\CD$ be finite braided $\CE$-modules. It follows directly from the definition that the evident equivalence $\CC \boxtimes_\CE \CD \simeq \CD \boxtimes_\CE \CC$ restricts to an equivalence 
	\begin{equation}
	\CC \boxdot_\CE \CD \simeq \CD \boxdot_\CE \CC. 
	\end{equation}
	Moreover, we have
	\begin{equation}
	(\CB \boxdot_\CE \CC) \boxdot_\CE \CD \simeq \CB \boxdot_\CE (\CC \boxdot_\CE \CD),
	\end{equation}
	because both sides of the equivalence are identified with the full subcategory of $\CB \boxtimes_\CE \CC \boxtimes_\CE \CD$ formed by the objects $x$ such that $\tau^\CB_{e,x} = \tau^\CC_{e,x} = \tau^\CD_{e,x}$ for all $e \in \CE$.
\end{rem}

\begin{rem} \label{rem:centralizer-SET}
	If $\CC$ is a finite braided $\CE$-module, then $\CE \boxdot_\CE \CC$ is the full subcategory of $\CC$ consisting of the objects $x$ such that $\tau_{e,x}=\Id_{e \odot x}$ for all $e\in\CE$. In particular, when $\CC$ is a finite braided monoidal category containing $\CE$, $\CE \boxdot_\CE \CC \simeq \CE'|_\CC$ is the centralizer of $\CE$ in $\CC$. 
\end{rem}

\begin{rem} \label{rem:stacking-centralizer}
	Let $\CC,\CD$ be finite braided $\CE$-modules. Then we have
	\begin{equation} \label{eq:stacking-excitation}
	(\CE \boxdot_\CE \CC) \boxtimes_\CE (\CE \boxdot_\CE \CD) \simeq \CE \boxdot_\CE (\CC \boxdot_\CE \CD).
	\end{equation}
	Indeed, each side of the equivalence is  the full subcategory of $\CC \boxtimes_\CE \CD$ consisting of the object $x$ such that $\tau_{e,x}^\CC = \tau_{e,x}^\CD = \Id_{e \odot x}$ for all $e \in \CE$.
	In particular, if $\CC,\CD$ are finite braided monoidal categories containing $\CE$, the equivalence \ref{eq:stacking-excitation} is a braided monoidal equivalence.
\end{rem}

\begin{prop} \label{prop:rtp-unit}
	Let $\CC$ be a finite braided $\CE$-module. The functor $\FZ(\CE) \boxtimes_\CE \CC \to \CC$, $e \boxtimes_\CE x \mapsto e \odot x$ induces an equivalence 
	\begin{equation}
	\FZ(\CE) \boxdot_\CE \CC \simeq \CC.
	\end{equation}
\end{prop}

	Note that $\FZ(\CE) \boxtimes_\CE \CC \simeq \Fun_{\CE|\CE}(\CE,\CE) \boxtimes_\CE \CC \simeq \Fun_{\CE|\CE}(\CE,\CE \boxtimes_\CE \CC) \simeq \Fun_{\CE|\CE}(\CE,\CC)$.
	Therefore, before we give a proof of this proposition, we study $\Fun_{\CE|\CE}(\CE,\CC)$ first.
	The category $\Fun_{\CE|\CE}(\CE,\CC)$ has a natural structure of a category over $\CE$, which can be described as follows.
	An object is a pair: 
	$$(x,\gamma_{x,-}=\{x \odot e \stackrel{\gamma_{x,e}}{\longrightarrow} e \odot x\}_{e \in \CE})$$
	where $x \in \CC$ and $\gamma_{x,-}$ is a half-braiding (i.e. a natural isomorphism in $\CC$ in the variable $e \in \CE$ and satisfying $(\Id_{e} \odot \gamma_{x,e'}) \circ (\gamma_{x,e} \odot \Id_{e'})=\gamma_{x,e \otimes e'}$ and $\gamma_{x,1}=\Id_x$). A morphism $(x,\gamma) \to (x',\gamma')$ is defined by a morphism $f:x \to x'$ preserving half-braidings.

\begin{proof}[Proof of Proposition \ref{prop:rtp-unit}]
	The composed equivalence
	$$\FZ(\CE) \boxtimes_\CE \CC \simeq \Fun_{\CE|\CE}(\CE,\CE) \boxtimes_\CE \CC \simeq \Fun_{\CE|\CE}(\CE,\CE \boxtimes_\CE \CC) \simeq \Fun_{\CE|\CE}(\CE,\CC)$$
	carries $(e',c_{e',-}) \boxtimes_\CE y \mapsto (e' \otimes -,c_{e',-}) \boxtimes_\CE y \mapsto (e' \otimes -,c_{e',-}) \boxtimes_\CE y \mapsto (e' \odot y,\gamma_{e' \odot y,-})$,	where $\gamma_{e' \odot y,-} = \{ e' \odot y \odot - \xrightarrow{\Id_{e'} \odot \beta_{y,-}} e' \odot - \odot y \xrightarrow{c_{e',-} \odot \Id_y} - \odot e' \odot y \}$. In addition, this composed equivalence preserves the $\CE$-action and the $\CE$-module braiding. Then, we have the following calculation (where we omit the module constraint):
	$$
	\begin{aligned}
	\tau^{\FZ(\CE)}_{e,e' \odot y}
	&=\tau^{\FZ(\CE)}_{e,e'} \odot \Id_y \\
	&=(c_{e',e} \circ \beta_{e,e'}) \odot \Id_y\\
	&=(c_{e',e} \odot \Id_y) \circ (\beta_{e,e'} \odot \Id_y)\\
	&=(c_{e',e} \odot \Id_y) \circ (\Id_{e'} \odot b_{y,e}) \circ (\Id_{e'} \odot b_{y,e}^{-1}) \circ (\beta_{e,e'} \odot \Id_y)\\
	&=\gamma_{e' \odot y,e} \circ b_{e' \odot y,e}^{-1}.\\
	\end{aligned}
	$$
	Consequently, we have $\tau^{\FZ(\CE)}_{e,x}=\gamma_{x,e} \circ b_{x,e}^{-1}$, for $(x,\gamma_{x,-}) \in \Fun_{\CE|\CE}(\CE,\CC)$ and $e \in \CE$. Meanwhile, the $\CE$-module braiding induced by that of $\CC$ is exactly $\tau^\CC_{e,x}$.
	Thus, an object $(x,\gamma_{x,-}) \in \FZ(\CE) \boxdot_\CE \CC$ amounts to $\gamma_{x,e}=\tau_{e,x}^\CC \circ b_{x,e}$ for all $e \in \CE$.
	
	To prove the restriction of the functor $\FZ(\CE) \boxtimes_\CE\CC \to \CC$, $(e,c_{e,-}) \boxtimes_\CE x \mapsto e\otimes x$ to $\FZ(\CE) \boxdot_\CE \CC$ (denoted by $F$) is an equivalence, it amounts to show it is essentially surjective and fully faithful.
	
	Given an object $x \in \CC$, we point out its essential preimage is $(x,\gamma_{x,-})$ where $\gamma_{x,e}=\tau_{e,x}^\CC \circ b_{x,e}$ for all $e \in \CE$. 
	This means the functor $F$ is essentially surjective. 

	Given two objects $(x,\gamma_{x,-}),(x',\gamma'_{x,-}) \in \FZ(\CE)\boxdot_\CE\CC$ and a morphism $f\in \Hom_\CC(x,x')$, we have the following commutative diagram by the naturality of $\tau^\CC$:
	$$
	\xymatrix@C=48pt{
		x \odot e \ar[r]^-{b_{x,e}}_{=} \ar[d]_{f \odot \Id_{e}} & e \odot x \ar[r]^-{\tau_{e,x}^\CC} \ar[d]^{\Id_{e} \odot f} & e \odot x \ar[d]^{\Id_{e} \odot f}\\
		x' \odot e \ar[r]^-{b_{x',e}}_{=} & e \odot x' \ar[r]^-{\tau_{e,x'}^\CC} & e \odot x'\\
	}
	$$
	for any $e \in \CE$.
	In this diagram, the composition of the top horizontal arrows is $\gamma_{x,e}$ and the composition of the bottom horizontal arrows is $\gamma_{x',e}$. This means $f$ is compatible with $\gamma$ and $\gamma'$, which shows that the functor $F$ is fully faithful. 
\end{proof}

\begin{prop} \label{prop:rtp-Fun'}
	Let $\CC,\CD$ be finite braided $\CE$-modules. Then the equivalence $\CC^\op \boxtimes_\CE\CD \simeq \Fun_\CE(\CC,\CD)$ restricts an equivalence 
	\begin{equation}
	\CC^\op \boxdot_\CE \CD \simeq \Fun_\CE'(\CC,\CD),
	\end{equation}
	where $\Fun_\CE'(\CC,\CD)$ is the full subcategory of $\Fun_\CE(\CC,\CD)$ formed by the braided $\CE$-module functors (see Definition\ref{def:Fun'}).
\end{prop}

\begin{proof}
	Suppose $\CC=\RMod_A(\CE),\CD=\RMod_B(\CE)$ for algebras $A,B \in \CE$ (see Proposition \ref{prop:rec-mod}). Recall that the equivalence (denoted by $G$) 
	$$\CC^\op \boxtimes_\CE \CD \simeq \Fun_\CE(\CC,\CD)$$ 
	carries $x \boxtimes_\CE y \mapsto - \otimes_A x^R \otimes y$ (see Theorem \ref{prop:rec-tp}). 
	
	Then let us see to which maps the $\CE$-module braidings are carried via $G$.
	
	Pick objects $x \in \CC, y \in \CD$ and $e \in \CE$.
	On the one hand, we have $G(\tau^{\CC^\op}_{e,x \boxtimes_\CE y}) = G(\tau^{\CC^\op}_{e,x} \boxtimes_\CE \Id_y) = \Id_- \otimes_A \tau^\CC_{e,x^R} \otimes \Id_y$. Note that given $x',x'' \in \CC$, $x'$ is also a left $A$-module, and we have $\Id_{x'} \otimes_A \tau^\CC_{e,x''} = \tau^\CC_{e,x'} \otimes_A \Id_{x''}$. It is not hard to verify this by using the formula $\tau^\CC_{e,x} = \tau^\CC_{e,A} \otimes_A \Id_x$ (see Example \ref{exam:ModA-br}).
	Thus, we have $$G(\tau^{\CC^\op}_{e,x \boxtimes_\CE y}) = \tau^\CC_{e,-} \otimes_A \Id_{x^R \otimes y}.$$
	
	On the other hand, we have $G(\tau^{\CD}_{e,x \boxtimes_\CE y}) = \Id_{x} \boxtimes_\CE G(\tau^{\CD}_{e,y}) = \Id_{- \otimes_A x^R} \otimes \tau^\CD_{e,y}$. Also note that $- \otimes_A x^R$ is indeed an object in $\CE$.
	Thus, we have $$G(\tau^{\CD}_{e,x \boxtimes_\CE y}) = \tau^\CD_{e,- \otimes_A x^R \otimes y}.$$
	
	As a consequence, picking $z \in \CC^\op \boxtimes_\CE \CD$ and obtaining $G(z) \in \Fun_\CE(\CC,\CD)$, we have	
	$$G(\tau^{\CC^\op}_{e,z}) = G(z)(\tau^\CC_{e,-}) \quad {\rm and} \quad G(\tau^{\CD}_{e,z}) = \tau^\CC_{e,G(z)(-)}$$
	for all $e \in \CE$, which means $z \in \CC^\op \boxdot_\CE \CD$ amounts to $G(z) \in \Fun'_\CE(\CC,\CD)$ and leads to the conclusion.
	
\end{proof}

\begin{cor} \label{cor:op-br}
	Let $\CC$ be a finite braided $\CE$-module. We have a canonical equivalence
	\begin{equation}
	\CC^\op \simeq \Fun_\CE'(\CC,\FZ(\CE)).
	\end{equation}
\end{cor}

\begin{rem} \label{rem:preseve-br}
	It can be checked from the definitions routinely that the equivalences in Remark \ref{rem:commutation-associative}, Proposition \ref{prop:rtp-unit}, \ref{prop:rtp-Fun'} and Corollary \ref{cor:op-br} preserve the $\CE$-module braidings. 
\end{rem}

\begin{cor} \label{cor:funa}
	Let $\CA$ be a finite monoidal category, $_\CA\CC$ be a left $\CA$-module, and $\CD,\CB$ be finite braided $\CE$-modules. Suppose $\CD$ is equipped with a right exact $k$-linear monoidal functor $\CA\to\Fun_\CE'(\CD,\CD)$. Then the equivalence $\Fun_\CA(\CC,\CD)\boxtimes_\CE\CB \simeq \Fun_\CA(\CC,\CD\boxtimes_\CE\CB)$, $f\boxtimes_\CE x \mapsto f\boxtimes_\CE x$ restricts to an equivalence 
	\begin{equation}
	\Fun_\CA(\CC,\CD)\boxdot_\CE\CB \simeq \Fun_\CA(\CC,\CD\boxdot_\CE\CB)
	\end{equation}
	where the $\CE$-action and the $\CE$-module braiding on $\Fun_\CA(\CC,\CD)$ are induced by those on $\CD$.
\end{cor}

\begin{proof}
	The condition that $\CD$ is a finite $\CE$-module equipped with a right exact $k$-linear monoidal functor $\CA\to\Fun_\CE'(\CD,\CD)$ means that $\CD$ is a finite $\CA$-$\CE$-bimodule and the left $\CA$-action on $\CD$ preserves the $\CE$-module braiding. Thus, we have the following composed equivalence:
	$$
	\begin{aligned}
	\Fun_\CA(\CC,\CD) \boxdot_\CE \CB 
	&\simeq \Fun'_\CE(\CB^\op,\Fun_\CA(\CC,\CD)) \\
	&\simeq \Fun_\CA(\CC,\Fun'_\CE(\CB^\op,\CD))
	\simeq \Fun_\CA(\CC,\CD \boxdot_\CE \CB).
	\end{aligned}
	$$
\end{proof}
	
	Let $\CC,\CD$ be finite monoidal categories over $\CE$, and $_\CC\CM_\CD,_\CC\CN_\CD$ finite $\CC$-$\CD$-bimodules over $\CE$. Considering the category $\Fun_{\CC|\CD}(\CM,\CN)$ and its $\CE$-module braiding mentioned in Example \ref{exam:fun_CD-br1}, we have the following useful fact:
\begin{lem}  \label{lem:fun_cd}
	We have 
	$\CE \boxdot_\CE \Fun_{\CC|\CD}(\CM,\CN)
	\simeq \Fun_{\CC\boxtimes_\CE\CD^\rev}(\CM,\CN)$.
\end{lem}

\begin{proof}
	Note that each side of the equivalence is a full subcategory of the category $\Fun_{\CC|\CD}(\CM,\CN)$. Then, given a functor $(f,S^{f|L},S^{f|R}) \in \Fun_{\CC|\CD}(\CM,\CN)$, $f$ is in $\Fun_{\CC \boxtimes_\CE \CD^\rev}(\CM,\CN)$ if and only if the diagram
	$$
	\xymatrix@C=48pt{
		\phi_\CC(e) \odot f(-) \ar[r]^-{\eta_{\CN,e} \circ \Id_f} \ar[d]_{S^{f|L}_{\phi_\CC(e),-}} & f(-) \odot \phi_\CD(e) \ar[d]^{S^{f|R}_{-,\phi_\CD(e)}} \\
		f(\phi_\CC(e) \odot -) \ar[r]^-{\Id_f \circ \eta_{\CM,e}} & f(- \odot \phi_\CD(e)) \\
	}
	$$
	commutes for all $e \in \CE$.
	
	Move clockwise around this square and compose these isomorphisms from the top left, and then we get the following composed morphism:
	$$
	\begin{aligned}
	\phi_\CC(e) \odot f(-) 
	&\xrightarrow{\eta_{\CN,e} \circ f} f(-) \odot \phi_\CD(e) 
	\xrightarrow{S^{f|R}_{-,\phi_\CD(e)}} f(- \odot \phi_\CD(e)) \\	
	&\xrightarrow{f \circ \eta_{\CM,e}^{-1}} f(\phi_\CC(e) \odot -) 
	\xrightarrow{(S^{f|L}_{\phi_\CC(e),-})^{-1}} \phi_\CC(e) \odot f(-), \\
	\end{aligned}
	$$
	and coincidentally it is nothing but the $\CE$-module braiding on $f$.
	Then the fact that the diagram commutes exactly means the $\CE$-module braiding is trivial, which leads to the target equivalence.
\end{proof}

	According to \cite[Lemma 3.1.1]{KZ1}, we have the equivalence $\Fun_{\CC|\CD}(\CM,\CN) \simeq \Fun_{\CC|\CC}(\CC,\Fun_{\CD^\rev}(\CM,\CN))$, where the right-$\CC$-module structure on $\Fun_{\CD^\rev}(\CM,\CN)$ is induced by the left-$\CC$-module structure of $\CM$. Then, we have the following fact:
\begin{cor} \label{cor:expand}
	We have $\Fun_{\CC\boxtimes_\CE\CD^\rev}(\CM,\CN) \simeq \CE\boxdot_\CE\Fun_{\CC|\CC}(\CC,\Fun_{\CD^\rev}(\CM,\CN))$. In particular, $\Fun_\CC(\CM,\CN) \simeq \CE\boxdot_\CE\Fun_{\CC|\CC}(\CC,\Fun_\CE(\CM,\CN))$. 
\end{cor}

	Sometimes, for a finite $\CE$-module $\CC$, there can be more than one nonequivalent $\CE$-module braiding ($\tau^{\CC|i},i=1,2,...$) on it. When we write down a relative tensor product over $\CE$ of $\CC$ and $\CD$ (another finite braided $\CE$-module), we need to point out which braiding we use in the condition $\tau^\CC_{e,x} = \tau^\CD_{e,x}$. 
	For convenience, in this paper, we use $\CC \boxdot_\CE^i \CD$ to denote the relative tensor product with respect to the condition $\tau^{\CC|i}_{e,x} = \tau^\CD_{e,x}$ for all $e \in \CE$.

\begin{exam} \label{exam:brs}
	Let $\CC,\CD$ be finite monoidal categories over $\CE$ and $_\CC\CM_\CD,_\CC\CN_\CD$ finite $\CC$-$\CD$-bimodules over $\CE$. By Corollary \ref{cor:expand}, there is an equivalence $\Fun_{\CC|\CD}(\CM,\CN) \simeq \CE \boxdot_\CE \Fun_{\CC|\CC}(\CC,\Fun_{\CE|\CD}(\CM,\CN))$. Then note that there are three braidings on the $\CE$-module $\Fun_{\CC|\CC}(\CC,\Fun_{\CE|\CD}(\CM,\CN))$ as follows:
	
	$
	\tau^1_{e,F} : \phi_\CC(e) \odot F(-)(-) 
	\xrightarrow{S^{F(-)|L}_{e,-}} F(-)(\phi_\CC(e) \odot -) 
	\xrightarrow{S^{F|R}_{-,\phi_\CC(e)}} F(- \odot \phi_\CC(e))(-) 
	\xrightarrow{\sim} F(\phi_\CC(e) \odot -)(-) 
	\xrightarrow{(S^{F|L}_{\phi_\CC(e),-})^{-1}} \phi_\CC(e) \odot F(-)(-) 
	$,
	
	$
	\tau^2_{e,F} : \phi_\CC(e) \odot F(-)(-) 
	\xrightarrow{\eta_{\CN,e}} F(-)(-) \odot \phi_\CD(e) 
	\xrightarrow{S^{F(-)|R}_{-,\phi_\CD(e)}} F(-)(- \odot \phi_\CD(e)) 
	\xrightarrow{\eta_{\CM,e}^{-1}} F(-)(\phi_\CC(e) \odot -) 
	\xrightarrow{S^{F|R}_{-,\phi_\CC(e)}} F(- \odot \phi_\CC(e))(-) 
	\xrightarrow{\sim} F(\phi_\CC(e) \odot -)(-) 
	\xrightarrow{(S^{F|L}_{\phi_\CC(e),-})^{-1}} \phi_\CC(e) \odot F(-)(-) 
	$ and
	
	$
	\tau^3_{e,F} : \phi_\CC(e) \odot F(-)(-) 
	\xrightarrow{\eta_{\CN,e}} F(-)(-) \odot \phi_\CD(e) 
	\xrightarrow{S^{F(-)|R}_{-,\phi_\CD(e)}} F(-)(- \odot \phi_\CD(e)) 
	\xrightarrow{\eta_{\CM,e}^{-1}} F(-)(\phi_\CC(e) \odot -) 
	\xrightarrow{(S^{F(-)|L}_{e,-})^{-1}} \phi_\CC(e) \odot F(-)(-) 
	$.
	
	Here, $S^{F|L}$, $S^{F|R}$ are the isomorphisms equipped by the $\CC$-$\CC$-bimodule functor $F: \CC \to \Fun_{\CE|\CD}(\CM,\CN)$, and $S^{F(-)|L}$, $S^{F(-)|R}$ are the isomorphisms equipped by the $\CE$-$\CD$-bimodule functor $F(-) \in \Fun_{\CE|\CD}(\CM,\CN)$, and $\eta_\CM,\eta_\CN$ are the natural transformations of finite $\CC$-$\CD$-bimodules $\CM,\CN$ over $\CE$. 
	
	Note that 
	\begin{equation} \label{eq:br-relation} 
	\tau^3_{e,F} = (\tau^1_{e,F})^{-1} \circ \tau^2_{e,F}.
	\end{equation}  
	
	For readers' convenience, we provide a view to see these three $\CE$-module braidings: 
	\begin{enumerate}
	\item $\tau^1$ is caused by the left-$\CC$-module and the $\CE$-module structures; 
	\item $\tau^2$ is caused by the left-$\CC$-module and the right-$\CD$-module structures; 
	\item $\tau^3$ is caused by the right-$\CD$-module and the $\CE$-module structures. 
	\end{enumerate}
	This is the reason why these three $\CE$-module braidings satisfy the relation \eqref{eq:br-relation}.
\end{exam}

\begin{thm} \label{thm:zfun}
	Let $\CC_1,\CD_1,\CC_2,\CD_2$ be finite monoidal categories over $\CE$, and let $_{\CC_1}\CM_{1\,\CD_1},_{\CC_1}\CN_{1\,\CD_1},_{\CC_2}\CM_{2\,\CD_2},{}_{\CC_2}\CN_{2\,\CD_2}$ be finite bimodules over $\CE$. Then the formula $f\boxdot_\CE g \mapsto f\boxtimes_\CE g$ determines an equivalence
	\begin{equation} \label{eq:stacking}
	\Fun_{\CC_1|\CD_1}(\CM_1,\CN_1) \boxdot_\CE \Fun_{\CC_2|\CD_2}(\CM_2,\CN_2) \simeq \Fun_{\CC_1\boxtimes_\CE\CC_2|\CD_1\boxtimes_\CE\CD_2}(\CM_1\boxtimes_\CE\CM_2,\CN_1\boxtimes_\CE\CN_2).
	\end{equation} 
	Moreover, when $\CM_1=\CN_1$ and $\CM_2=\CN_2$, \eqref{eq:stacking} is a monoidal equivalence.	
	In particular, the equivalence
	\begin{equation} \label{eq:center-monoidal}
	\FZ(\CC)\boxdot_\CE\FZ(\CD) \simeq \FZ(\CC\boxtimes_\CE\CD)
	\end{equation}
	is an equivalence of braided monoidal categories containing $\CE$.
\end{thm}

\begin{proof}
	We claim that we can reduce the problem (to show the equivalence \eqref{eq:stacking}) to the special case $\CC_1=\CD_1=\CC_2=\CD_2=\CE$. We give the process of making $\CC_1=\CE$, and it is similar for others.

	Note that we have the following composed equivalence:
	\begin{equation} \label{eq:reduction}
	\begin{aligned}
	\Fun_{\CC_1|\CD_1}&(\CM_1,\CN_1) \boxdot_\CE \Fun_{\CC_2|\CD_2}(\CM_2,\CN_2) \\
	&\simeq \left(\CE \boxdot_\CE^1 \Fun_{\CC_1|\CC_1}(\CC_1,\Fun_{\CE|\CD_1}(\CM_1,\CN_1))\right) \boxdot_\CE^2 \Fun_{\CC_2|\CD_2}(\CM_2,\CN_2) \\
	&\simeq \CE \boxdot_\CE^1 \left(\Fun_{\CC_1|\CC_1}(\CC_1,\Fun_{\CE|\CD_1}(\CM_1,\CN_1)) \boxdot_\CE^3 \Fun_{\CC_2|\CD_2}(\CM_2,\CN_2)\right) \\
	&\simeq \CE \boxdot_\CE^1 \Fun_{\CC_1|\CC_1}(\CC_1,\Fun_{\CE|\CD_1}(\CM_1,\CN_1) \boxdot_\CE \Fun_{\CC_2|\CD_2}(\CM_2,\CN_2)). \\
	\end{aligned}
	\end{equation}
	Here, we follow the notations in Example \ref{exam:brs} using $\tau^i$ to denote the $\CE$-module braidings on $\Fun_{\CC_1|\CC_1}(\CC_1,\Fun_{\CE|\CD_1}(\CM_1,\CN_1))$. And we use $\tau'$ to denote the $\CE$-module braiding on $\Fun_{\CC_2|\CD_2}(\CM_2,\CN_2)$ mentioned in Example \ref{exam:fun_CD-br1}.
	
	In the first step, we used Corollary \ref{cor:expand}. 
	
	In the second step, both sides of the equivalence can be identified with a full subcategory of $\Fun_{\CC_1|\CC_1}(\CC_1,\Fun_{\CE|\CD_1}(\CM_1,\CN_1)) \boxtimes_\CE \Fun_{\CC_2|\CD_2}(\CM_2,\CN_2)$. Note that the former consists	of the objects $x$ such that $\tau_{e,x}^1=\Id_{e \odot x}$ and $\tau_{e,x}^2=\tau'_{e,x}$, while the latter is formed by the objects $x$ such that $\tau_{e,x}^1=\Id_{e \odot x}$ and $\tau_{e,x}^3=\tau'_{e,x}$. These two conditions are equivalent due to the relation \eqref{eq:br-relation} $\tau^3_{e,F} = (\tau^1_{e,F})^{-1} \circ \tau^2_{e,F}$. 
	
	In the last step, we used Corollary \ref{cor:funa}, where the right exact $k$-linear monoidal functor $\CC_1 \boxtimes \CC_1^\rev \to \Fun'_\CE(\Fun_{\CE|\CD_1}(\CM_1,\CN_1),\Fun_{\CE|\CD_1}(\CM_1,\CN_1))$ is given by $c_1 \boxtimes c_2 \mapsto \{f(-) \mapsto c_1 \odot f(c_2 \odot -)\}_{f \in \Fun_{\CE|\CD_1}(\CM_1,\CN_1)}$.
	
	Then, we find that the problem is reduced to the special case $\CC_1=\CE$, i.e. to show the following equivalence: 	
	\begin{equation} \label{eq:assumption}
	\Fun_{\CE|\CD_1}(\CM_1,\CN_1) \boxdot_\CE \Fun_{\CC_2|\CD_2}(\CM_2,\CN_2) \simeq \Fun_{\CC_2|\CD_1\boxtimes_\CE\CD_2}(\CM_1\boxtimes_\CE\CM_2,\CN_1\boxtimes_\CE\CN_2).
	\end{equation} 
	In fact, if we have the equivalence \eqref{eq:assumption}, we can go on with our calculation following \eqref{eq:reduction}:
	\begin{equation}
	\begin{aligned}
	\Fun_{\CC_1|\CD_1}&(\CM_1,\CN_1) \boxdot_\CE \Fun_{\CC_2|\CD_2}(\CM_2,\CN_2) \\
	&\simeq \CE \boxdot_\CE^1 \Fun_{\CC_1|\CC_1}(\CC_1,\Fun_{\CE|\CD_1}(\CM_1,\CN_1) \boxdot_\CE \Fun_{\CC_2|\CD_2}(\CM_2,\CN_2)) \\
	&\simeq \CE \boxdot_\CE^1 \Fun_{\CC_1|\CC_1}(\CC_1,\Fun_{\CC_2|\CD_1\boxtimes_\CE\CD_2}(\CM_1\boxtimes_\CE\CM_2,\CN_1\boxtimes_\CE\CN_2))\\
	&\simeq \Fun_{\CC_1\boxtimes_\CE\CC_2|\CD_1\boxtimes_\CE\CD_2}(\CM_1\boxtimes_\CE\CM_2,\CN_1\boxtimes_\CE\CN_2) \\
	\end{aligned}
	\end{equation} 
	where we used Corollary \ref{cor:expand} inversely in the last step.
	
	By the same token, we can reduce the problem to the special case $\CC_1=\CD_1=\CC_2=\CD_2=\CE$.
	
	Also note that the equivalence ($i=1,2$)
	$$
	\Fun_{\CE|\CE}(\CM_i,\CN_i) \simeq \Fun_{\CE|\CE}(\CE,\CM_i^\op\boxtimes_\CE\CN_i) \simeq \FZ(\CE)\boxtimes_\CE\CM_i^\op\boxtimes_\CE\CN_i,
	$$
	carries the $\CE$-module braiding on $\Fun_{\CE|\CE}(\CM_i,\CN_i)$ to the one induced by $\FZ(\CE)$. 	
	Therefore, it suffices to show $\FZ(\CE) \boxdot_\CE \FZ(\CE) \simeq \FZ(\CE)$, which is a special case of Proposition \ref{prop:rtp-unit}.
	
	By a tedious but routine calculation, one can verify the equivalence \eqref{eq:stacking} carries $f \boxdot_\CE g \mapsto f \boxtimes_\CE g$, which is apparent in the special case $\FZ(\CE) \boxdot_\CE \FZ(\CE) \simeq \FZ(\CE)$.
	
	When $\CM_1=\CN_1$ and $\CM_2=\CN_2$, the formula $f \boxdot_\CE g \mapsto f \boxtimes_\CE g$ clearly defines a monoidal equivalence. Moreover, in \eqref{eq:center-monoidal}, this prescription preserves braidings.	
\end{proof}	

\begin{cor}
	Let $\CC,\CD$ be finite monoidal categories over $\CE$. Then there is a braided monoidal equivalence $$\FZ_{/\CE}(\CC) \boxtimes_\CE \FZ_{/\CE}(\CD) \simeq \FZ_{/\CE}(\CC \boxtimes_\CE \CD).$$
	In particular, when $\CE=\bk$, there is a braided monoidal equivalence $$\FZ(\CC)\boxtimes\FZ(\CD) \simeq \FZ(\CC\boxtimes\CD).$$
\end{cor}	

\begin{proof}
	Note that, by Remark \ref{rem:stacking-centralizer}, we have the composed braided monoidal equivalence $(\CE \boxdot_\CE \FZ(\CC)) \boxtimes_\CE (\CE \boxdot_\CE \FZ(\CD)) \simeq \CE \boxdot_\CE (\FZ(\CC) \boxdot_\CE \FZ(\CD)) \simeq \CE \boxdot_\CE \FZ(\CC \boxtimes_\CE \CD)$.
\end{proof}

\subsection{An alternative approach to the relative tensor product over $\CE$} \label{sec:rtensor_localm}

	In this subsection, we provide another description of the relative tensor product over $\CE$ in a restricted case. 

	Let $k$ be an algebraically closed field and $\CE$ a symmetric fusion category over $k$.
	Let $\CC$, $\CD$ be two braided multi-tensor categories fully faithfully containing $\CE$. 
	
	Let $R:\CE \to \CE \boxtimes \CE$ be the right adjoint functor of the tensor product functor $\otimes:\CE \boxtimes \CE \to \CE$. Set $L_\CE:=R(\one_\CE)$, which is the canonical algebra in $\CE \boxtimes \CE$ and has a decomposition as $L_\CE=\oplus_{i \in \CO(\CE)}i^L \boxtimes i$. Moreover, if $\CC,\CD$ are braided fusion categories, $L_\CE$ is a condensable algebra in $\CC \boxtimes \CD$. (See \cite{ENO3,DGNO,DMNO,DNO}.) We denote by $\RMod_{L_\CE}^0(\CC \boxtimes \CD)$ the full subcategory of $\RMod_{L_\CE}(\CC \boxtimes \CD)$ formed by local right $L_\CE$-modules in $\CC \boxtimes \CD$.

\begin{prop} \label{prop:rtp-localm}
	The equivalence $\RMod_{L_\CE}(\CC \boxtimes \CD) \simeq \CC \boxtimes_\CE \CD$ restricts to an equivalence
	$$\RMod_{L_\CE}^0(\CC \boxtimes \CD) \simeq \CC\boxdot_\CE \CD.$$
\end{prop}

\begin{proof}
	We follow the notations in the diagram \eqref{eq:tp-m}:
	\begin{equation*} 
	\xymatrix@C=48pt{
		\CM \boxtimes \CN \ar[r]^-{- \otimes L_\CC} \ar[d]|-{F:m \boxtimes n \mapsto m \boxtimes_\CC n} & \RMod_{L_\CC}(\CM \boxtimes \CN) \ar[ld]^-{\tilde{F}} \\
		\CM \boxtimes_\CC \CN. \\		
	}
	\end{equation*} 
	
	It is sufficient to show that when restricted to $\RMod_{L_\CE}^0(\CC\boxtimes\CD)$, the essential image of $\tilde{F}$ is exactly $\CC \boxdot_\CE \CD$, i.e. $\tilde{F}(\RMod_{L_\CE}^0(\CC\boxtimes\CD)) \simeq \CC \boxdot_\CE \CD$.
		
	On the one hand, for an object $y \in \CC \boxtimes_\CE \CD$, as we are discussing the fusion case, to judge if $y \in \CC \boxdot_\CE \CD$, it suffices to check the property $\tau^\CC_{j,y} = \tau^\CD_{j,y}$ only for simple objects $j \in \CE$.
	On the other hand, given an object $x \in \RMod_{L_\CE}(\CC\boxtimes\CD)$, we also expect to find what property satisfied by $\tilde{F}(x)$ is equivalent to $x \in \RMod_{L_\CE}^0(\CC\boxtimes\CD)$. For convenience, we denote $y=\tilde{F}(x)$.
	
	Consider the decomposition $L_\CE = \oplus_{j \in \CO(\CE)} j^L \boxtimes j$, where $\CO(\CE)$ denotes the set formed by all isomorphism classes of simple objects in $\CE$. 
	Then, for a $L_\CE$-module $x$ equipped with a $L_\CE$-action $\mu_x: x \odot L_\CE \to x$, there is a decomposition $\mu_x = \oplus_{j \in \CO(\CE)} \mu^j_x$, where $\mu^j_x$ is the morphism from $x \otimes (j^L \boxtimes j)$ to $x$. 
	Moreover, we also have a similar decomposition of double braidings $c_{L_\CE, x} \circ c_{x, L_\CE} = \oplus_{j \in \CO(\CE)} c_{j^L \boxtimes j, x} \circ c_{x, j^L \boxtimes j}$.
	
	This gives an equivalent condition of $x \in \RMod_{L_\CE}^0(\CC\boxtimes\CD)$: the diagram
	$$
	\xymatrix@C=48pt{
		x \otimes (j^L \boxtimes j) \ar[rr]^{c_{j^L \boxtimes j, x} \circ c_{x, j^L \boxtimes j}} \ar[rd]_{\mu^j_x} & & x \otimes (j^L \boxtimes j) \ar[ld]^{\mu^j_x} \\
		& x & \\		
	}
	$$
	commuting for all simple objects $j$ in $\CE$.
	
	Then take the functor $\tilde{F}$. Considering a free module $(c \boxtimes d) \otimes L_\CE$,  we have 
	$$
	\begin{aligned}
	\tilde{F}\left((c \boxtimes d) \otimes L_\CE \otimes (j^L \boxtimes j)\right) &\cong \tilde{F}\left(((c \otimes j^L) \boxtimes (j \otimes d)) \otimes L_\CE\right) \cong (c \otimes j^L) \boxtimes_\CE (j \otimes d) \\
	&\cong (j^L \otimes j) \otimes c \boxtimes_\CE d \cong (j^L \otimes j) \otimes \tilde{F}\left((c \boxtimes d) \otimes L_\CE\right). \\
	\end{aligned}
	$$
	Hence, we get $\tilde{F}(x \otimes(j^L \boxtimes j)) \cong (j^L \otimes j) \otimes y$. Moreover, note that $\tilde{F}$ carries the double braiding $c_{j^L \boxtimes j, x} \circ c_{x, j^L \boxtimes j}$ to $\tau^\CC_{j^L,j \otimes y} \circ (\Id_{j^L} \otimes \tau^\CD_{j,y})$ and the morphism $\mu^j_x$ to $\ev_j \otimes \Id_y$. 
	Thus, we get the following commutative diagram:
	$$
	\xymatrix@C=48pt{
		(j^L \otimes j) \otimes y \ar[r]^{\Id_{j^L} \otimes \tau^\CD_{j,y}} \ar[rd]_{\ev_j \otimes \Id_y} & (j^L \otimes j) \otimes y \ar[r]^{\tau^\CC_{j^L,j \otimes y}} & (j^L \otimes j) \otimes y \ar[ld]^{\ev_j \otimes \Id_y} \\
		& y & .\\		
	}
	$$
	In other words,
	\begin{equation} \label{111}
	(\ev_j \otimes \Id_y) \circ (\Id_{j^L} \otimes (\tau^\CD_{j,y})^{-1}) = (\ev_j \otimes \Id_y) \circ \tau^\CC_{j^L,j \otimes y}.
	\end{equation} 
	
	Note that the equation \eqref{111} holding for all simple objects $j \in \CO(\CE)$ is a equivalent condition of $x \in \RMod_{L_\CE}^0(\CC\boxtimes\CD)$, since $\tilde{F}$ is fully faithful.
	
	We remind readers that there is a commutative diagram by the naturality of $\tau$:
	$$
	\xymatrix@C=48pt{
		(j^L \otimes j) \otimes y \ar[r]^{\Id_{j^L} \otimes \tau_{j,y}} \ar[d]_{\ev_j \otimes \Id_y} & (j^L \otimes j) \otimes y \ar[r]^{\tau_{j^L,j \otimes y}} & (j^L \otimes j) \otimes y \ar[d]^{\ev_j \otimes \Id_y} \\
		\one_\CE \otimes y \ar[rr]^{\tau_{\one_\CE,y}=\Id_{\one_\CE \otimes y}} & & \one_\CE \otimes y.\\		
	}
	$$	
	Similarly, we get that
	\begin{equation} \label{222}
	(\ev_j \otimes \Id_y) \circ (\Id_{j^L} \otimes (\tau_{j,y})^{-1}) = (\ev_j \otimes \Id_y) \circ \tau_{j^L,j \otimes y}.
	\end{equation} 	

	Take the braiding in the equation \eqref{222} as the one induced by the braiding in $\CC$ and consider the equation \eqref{111}.
	Then, we get 
	\begin{equation} \label{333}
	(\ev_j \otimes \Id_y) \circ (\Id_{j^L} \otimes (\tau^\CC_{j,y})^{-1}) = (\ev_j \otimes \Id_y) \circ (\Id_{j^L} \otimes (\tau^\CD_{j,y})^{-1}).
	\end{equation} 
	Note that the equation \eqref{333} is equivalent to $\tau^\CC_{j,y} = \tau^\CD_{j,y}$ due to the properties of the dual objects.
	Then, we get that $\{\tau^\CC_{j,y} = \tau^\CD_{j,y}\}_{\forall j \in \CO(\CE)}$ is also a equivalent condition of $x \in \RMod_{L_\CE}^0(\CC\boxtimes\CD)$, which completes the proof.
\end{proof}

\begin{rem}
	Proposition \ref{prop:rtp-localm} and Remark \ref{rem:centralizer-SET} directly imply that our definition of the relative tensor product over $\CE$ coincides with the one in \cite{LKW1}.% and is more general. 
\end{rem}

%%%%%%%%%%%%%%%%%%%%%%%%%%%%%%%%%%%%%%%%%%%%%%%%%%%%%%%%

\section{Symmetry enriched center functor} \label{sec:SE-center}

%%%%%%%%%%%%%%%%%%%%%%%%%%%%%%%%%%%%%%%%%%%%%%%%%%%%%%%%
	
	In this section, we assume $k$ is an algebraically closed field and let $\CE$ be a symmetric fusion category over $k$.
	
\subsection{Functoriality of the symmetry enriched center}
	In this subsection, we assume that the global dimension of $\CE$ does not vanish. Equivalently, the Drinfeld center of $\CE$ is semisimple (see \cite{KZ2}). When $\chara k=0$, this condition is redundant (see \cite{ENO1}).

\begin{lem} \label{lem:Rmod-ss}
	Let $G$ be a finite group whose order is not divided by $\chara k$. Let $A$ be an algebra in $\Rep G$. If $A$ is a semisimple $k$-algebra, then $\RMod_A(\Rep G)$ is semisimple.
\end{lem}

\begin{proof}
	An algebra $A$ in $\Rep G$ can be seen as an algebra $A$ in $k$ with a $G$-action, i.e. for any $g \in G$, $g:A \to A$ is an algebra homomorphism. An object $V$ in $\RMod_A(\Rep G)$ can be seen as an object in $\RMod_A(k)$ with a $G$-action, i.e. $g \circ a = g(a) \circ g$ for any $g \in G$ and $a \in A$.
	Then, a subobject $W$ of $V$ in $\RMod_A(\Rep G)$ is a right $A$-submodule stable under $G$-action.
	
	To prove the conclusion, it is equivalent to showing that for any subobject $W$ of $V$ in $\RMod_A(\Rep G)$, there exists a subobject $U$ of $V$ such that $V = W \oplus U$.
	We see $V$ as a right $A$-module over $k$ and $W$ as the submodule of $V$. Since $A$ is a semisimple $k$-algebra, we have $V = W \oplus U'$ for some right $A$-module $U'$. Let $p$ be the corresponding projection of $V$ onto $W$ as right $A$-modules. Then, we have $a \circ p = p \circ a$ for any $a \in A$.
	Form the average $p^0$ of the conjugates of $p$ by the elements of $G$:
	$$p^0=\frac{1}{|G|}\sum_{g \in G} g \circ p \circ g^{-1}.$$
	Then, we have $p^0$ is a projection of $V$ onto $W$ as well, since $p$ maps $V$ onto $W$ and $g$ preserves $W$ and for $x \in W$ we have $p^0 \ x = x$. 
	And we have $a \circ p^0 = p^0 \circ a$ for any $a \in A$, because of $g \circ p \circ g^{-1} \circ a = g \circ p \circ g^{-1}(a) \circ g^{-1} = g \circ g^{-1}(a) \circ p \circ g^{-1} = a \circ g \circ p \circ g^{-1}$ for any $g \in G$.
	Moreover, we have $g \circ p^0 = p^0 \circ g$ for any $g \in G$.
	Let $U = \ker p^0$. This shows $U$ is a right $A$-submodule stable under $G$-action, and $V = W \oplus U$, and completes the proof.
\end{proof}

\begin{lem} \label{lem:ss}
	Let $\CC $ be a finite $\CE$-module, and let $\CC^\sharp$ be the full subcategory of $\CC$ formed by the semisimple objects. Then $\CC^\sharp$ is an $\CE$-submodule. 
\end{lem}

\begin{proof}
	According to \cite[Corollaries 0.7 and 0.8]{De2} for $\chara k=0$ and \cite[Theorem 5.4]{EG} for $\chara k>0$, $\CE$ is equivalent to the representation category $\Rep G$ of a finite group $G$ over $k$, possibly with a modified symmetric structure. Moreover, $\chara k$ does not divide the order of $G$ (i.e. $\dim(\CE) \ne 0$). We may simply assume $\CE=\Rep G$.
	
	Suppose $\CC=\RMod_A(\CE)$ where $A$ is an algebra in $\CE$ i.e. a $k$-algebra with a $G$-action. Let $A'=A/R(A)$ where $R(A)$ is the radical of $A$. Then $A'$ is also an algebra in $\CE$. According to Lemma \ref{lem:Rmod-ss}, $\RMod_{A'}(\CE)\subset\CC^\sharp$. Conversely, if $x$ is a simple (or semisimple) right $A$-module in $\CE$, then $x R(A)$ as a proper submodule of $x$ in $\CE$ has to vanish, which means $x$ is a right $A'$-module in $\CE$. Therefore, $\CC^\sharp \subset \RMod_{A'}(\CE)$. This shows that $\CC^\sharp = \RMod_{A'}(\CE)$, which leads to that $\CC^\sharp$ is a left $\CE$-submodule of $\CC$. 
	
\end{proof}

\begin{cor} \label{cor:ss}
	Let $\CC,\CD$ be finite $\CE$-modules, and let $\CC^\sharp,\CD^\sharp$ be the full subcategory of $\CC,\CD$ formed by the semisimple objects, respectively. Then $\CC^\sharp\boxtimes_\CE\CD^\sharp$ is a semisimple full subcategory of $\CC\boxtimes_\CE\CD$.
\end{cor}

\begin{proof}
	According to \cite[Theorem 2.18]{ENO1}, $\CC^\sharp \boxtimes_\CE \CD^\sharp \simeq \Fun_\CE({\CC^{\sharp}}^{\op|L},\CD^\sharp)$ is semisimple.
\end{proof}

\begin{cor} \label{cor:ffb}
	Let $\CC,\CD$ be finite braided monoidal categories fully faithful containing $\CE$. Then $\CC\boxdot_\CE\CD$ is also fully faithful containing $\CE$.
\end{cor}

\begin{proof}
	By definition, $\phi_\CC:\CE\to\CC$ and $\phi_\CD:\CE\to\CD$ are fully faithful and carry simple objects to simple objects. Then, we have the composed braided monoidal functor $\CE \simeq \CE \boxdot_\CE \CE \xrightarrow{\phi_\CC \boxdot_\CE \phi_\CD} \CC \boxdot_\CE \CD$ is fully faithful and carries simple objects to simple objects.
\end{proof}

\begin{cor} \label{cor:ffm}
	Let $\CC,\CD$ be finite monoidal categories fully faithful over $\CE$. Then $\CC\boxtimes_\CE\CD$ is also fully faithful over $\CE$. 
\end{cor}

\begin{proof}
	$\FZ(\CC\boxtimes_\CE\CD)$ is fully faithful containing $\CE$ by Corollary \ref{cor:ffb} and Theorem \ref{thm:zfun}.
\end{proof}

\begin{thm} \label{thm:rigid}
	Let $\CC,\CD$ be multi-tensor categories over $\CE$. Then $\CC\boxtimes_\CE\CD$ is rigid.
\end{thm}

\begin{proof}
	According to Corollary \ref{cor:ss}, $a\boxtimes_\CE b$ is semisimple for simple objects $a\in\CC,b\in\CD$. Then apply \cite[Theorem 2.7.5]{KZ1}, which leads to the conclusion directly.
\end{proof}

\begin{cor} \label{cor:rigid}
	Let $\CC,\CD$ be braided multi-tensor categories containing $\CE$. Then $\CC\boxdot_\CE\CD$ is rigid.
\end{cor}

\begin{proof}
	We have $\CC\boxdot_\CE\CD \simeq \Fun'_\CE(\CC^\op,\CD)$ is the full subcategory of the rigid monoidal category $\CC\boxtimes_\CE\CD \simeq \Fun_\CE(\CC^\op,\CD)$. Note that $f^L$ preserves the $\CE$-module braidings if and only if $f$ preserves the $\CE$-module braidings, for $f \in \Fun_\CE(\CC^\op,\CD)$.
\end{proof}

	We introduce two categories $\mtc_{/\CE}$ and $\bmtc_\CE$ as follows:
	\begin{enumerate}
	\item The category $\mtc_{/\CE}$ of multi-tensor categories fully faithful over $\CE$ with morphisms given by the equivalence classes of finite bimodules over $\CE$. %The tensor product is $\boxtimes_\CE$.
	\item The category $\bmtc_\CE$ of braided tensor categories fully faithful containing $\CE$ with morphisms given by the equivalence classes of monoidal bimodules containing $\CE$. 
	\end{enumerate}
	According to \cite[Corollary 2.3.11]{KZ1}, Theorem \ref{thm:rigid} and Corollary \ref{cor:ffm}, the former category is a symmetric monoidal category under $\boxtimes_\CE$ with the tensor unit $\CE$. According to Proposition \ref{prop:rtp-finite}, Corollary \ref{cor:rigid}, Corollary \ref{cor:ffb}, Proposition \ref{prop:rtp-unit} and Remak \ref{rem:commutation-associative}, the latter category is also a symmetric monoidal category under $\boxdot_\CE$ with the tensor unit $\FZ(\CE)$.

\begin{thm} \label{thm:functorial_E}
	The assignment $$\CC \mapsto \FZ(\CC), \ {}_\CC\CM_\CD \mapsto \FZ^{(1)}(\CM) \simeq \Fun_{\CC|\CD}(\CM,\CM)$$ defines a symmetric monoidal functor $$\FZ: \mtc_{/\CE} \to \bmtc_\CE.$$
\end{thm}

\begin{proof}
	According to Example \ref{exam:fun_CD-br2}, $\FZ$ is well-defined on morphism. According to Theorem \ref{thm:functorial}, $\FZ$ is a well-defined functor. According to Theorem \ref{thm:zfun}, $\FZ$ is a symmetric monoidal functor.
\end{proof}
	We refer to this functor $\FZ$ as the \textit{symmetry enriched center functor}.
	
\begin{rem}
	For Theorem \ref{thm:functorial_E}, $\CE$ has to be restricted to a fusion category rather than a multi-fusion one (see Remark \ref{rem:E-ten-ind}) and the domain of $\FZ$ needs to be restricted to multi-tensor categories fully faithful over $\CE$ to ensure they are indecomposable (to use \cite[Theorem 3.1.7]{KZ1}).
	
	Moreover, in the special case $\CE=\bk$, a multi-tensor category is fully faithful over $\CE$ if and only if it is indecomposable; a braided tensor category is always fully faithful containing $\CE$. Therefore, Theorem \ref{thm:functorial_E} generalizes the usual case in \cite[Theorem 3.1.8]{KZ1}.
\end{rem}

\begin{rem}
	Consider the composition of Drinfeld center $\FZ(-)$ and $\CE \boxdot_\CE -$. It is nothing but the $/\CE$-center $\FZ_{/\CE}$. We can make $\FZ_{/\CE}$ a symmetric monoidal functor as well.
	
	More precisely, we introduce a category $\bmtc_{/\CE}$ of braided tensor categories fully faithful over $\CE$ with morphisms given by the equivalence classes of monoidal bimodules containing $\CE$. And then, the assignment
	$$\CC \mapsto \FZ_{/\CE}(\CC), \ {}_\CC\CM_\CD \mapsto \FZ^{(1)}_{/\CE}(\CM) := \Fun_{\CC \boxtimes_\CE \CD^\rev}(\CM,\CM)$$ defines a symmetric monoidal functor $\FZ_{/\CE}: \mtc_{/\CE} \to \bmtc_{/\CE}$, as well.
	
\end{rem}

\subsection{Fully-faithfulness of the symmetry enriched center}
	In this subsection, let $k$ be an algebraically closed field of characteristic zero and let $\CE$ be a symmetric fusion category over $k$. % (or require multi-fusion categories have semisimple Drinfeld center).
	
	And in what follows, we only consider braided fusion categories fully faithful containing $\CE$. Thus, we omit "fully faithful" for brevity.

	For two braided fusion categories $\CC,\CD$ containing $\CE$, we study the category $\CC \boxtimes_\CE \CD$. By Remark \ref{rem:tp-m}, we can identify $\CC \boxtimes_\CE \CD$ with $\RMod_{L_\CE}(\CC \boxtimes \CD)$, where it is worthwhile to note that $L_\CE$ is a condensable algebra in $\CC \boxtimes \CD$. According to \cite{DMNO}, for a braided fusion category $\CB$ and a condensable  algebra $A$ in $\CB$, the unit object of $\RMod_A(\CB)$ is $A$ and its connectedness implies that $A \in \RMod_A(\CB)$ (as a multi-fusion category) is simple, which leads to the consequence that $\RMod_A(\CB)$ is a fusion category. As a consequence, $\CC \boxtimes_\CE \CD$ is a braided fusion category. Furthermore, $\CC \boxdot_\CE \CD$, as a full subcategory of $\CC \boxtimes_\CE \CD$, respects fusion rules and braidings, and is closed under taking direct sums, thus is a braided fusion category containing $\CE$.
	
\begin{rem}
	There are several equivalent conditions of the nondegeneracy of a braided fusion category in \cite{DGNO}. One is that, for a braided fusion category $\CC$, its M\"uger center is trivial, i.e. $\CC' = \bk$.
\end{rem}

\begin{lem} \label{lem:non-dege}
	Let $\CC,\CD$ be braided fusion categories containing $\CE$, then the fact that $\CC \boxdot_\CE \CD$ is nondegenerate implies that $\CC$ and $\CD$ are nondegenerate.
\end{lem}

\begin{proof}
	Otherwise, we assume that there exists a simple $x \in \CC'$ nonisomorphic to $\one_\CC$. Then note that $\tau^\CC_{e,x}=\Id_{e \odot x}$ for any $e \in \CE$, which leads to $x \boxdot_\CE \one_\CD \in (\CC \boxdot_\CE \CD)'$. Moreover, $x \boxdot_\CE \one_\CD$ is simple and nonisomorphic to $\one_{\CC \boxdot_\CE \CD}$. These cause that $(\CC \boxdot_\CE \CD)'$ is nontrivial, which contradicts with that $\CC \boxdot_\CE \CD$ is nondegenerate.
\end{proof}		

\begin{prop} \label{prop:mfb-cl}
	Let $\CC,\CD$ be braided fusion categories containing $\CE$, and let $\CM$ be a multi-fusion $\CC$-$\CD$-bimodule containing $\CE$. Then $\CM$ is closed if and only if $\bar\CC\boxdot_\CE\CD \simeq \FZ(\CE\boxdot_\CE\CM)$.
\end{prop}

	Before we proof this proposition, we refer readers to considering the $\CE$-module braidings on $\CW:=\Fun_{\CE|\CM}(\CM,\CM)$. 

	Note that we have two ways to see $\FZ(\CM)$ as a braided fusion category containing $\CE$ (recall Definition \ref{def:bmccE}), i.e. there are two braided monoidal functors from $\CE$ to $\FZ(\CM)$, $\Phi_\CM^1:\CE \xrightarrow{\overline{\phi_\CC} \boxtimes \one_\CD} \overline{\CC} \boxtimes \CD \xrightarrow{\psi_\CM} \FZ(\CM)$ and $\Phi_\CM^2:\CE \xrightarrow{\one_{\bar\CC} \boxtimes \phi_\CD} \overline{\CC} \boxtimes \CD \xrightarrow{\psi_\CM} \FZ(\CM)$. The braided monoidal functor $\Phi_\CE:\CE \to \FZ(\CE)$ is evident. Then, we have $\pi_\CE \circ \Phi_\CE = \Id_\CE$, $\pi_\CM \circ \Phi_\CM^1 = \psi_\CM^-$ and $\pi_\CM \circ \Phi_\CM^2 = \psi_\CM^+$.

	All these lead that there are two ways to see $\CW$ as a monoidal $\FZ(\CE)$-$\FZ(\CM)$-bimodule containing $\CE$ as follows (recall Definition \ref{def:mbimcE}). We have a braided monoidal functor $\psi_\CW: \overline{\FZ(\CE)} \boxtimes \FZ(\CM) \simeq \overline{\FZ(\CE)}\boxtimes\overline{\FZ(\CM^\rev)} \to \FZ(\CW)$ such that $\psi_\CW(e \boxtimes m) = \psi_\CM^-(e) \otimes - \otimes m$ as objects of $\CW$, equipped with the evident half-braiding $\psi_\CM^-(e) \otimes F(-) \otimes m \simeq F(\psi_\CM^-(e) \otimes - \otimes m)$, where $e \in \overline{\FZ(\CE)}$, $m \in \overline{\FZ(\CM^\rev)}$, $F \in \Fun_{\CE|\CM}(\CM,\CM)$. 
	And we have two isomorphisms of functors (i.e. natural transformations) $\eta_{\CW,e}^1:\psi_\CM^-(e) \otimes - \to - \otimes \psi_\CM^-(e)$ and $\eta_{\CW,e}^2 = (\Id_- \otimes \eta_{\CM,e}) \circ \eta_{\CW,e}^1 : \psi_\CM^-(e) \otimes - \to - \otimes \psi_\CM^+(e)$. 

	Then, we have two evident $\CE$-module braidings on $\CW$ as follows: 

	$\tau_{e,F}^1:\psi_\CM^-(e) \otimes F(-) \xrightarrow{\eta_{\CW,e}^1 \circ \Id_F} F(-) \otimes \psi_\CM^-(e) \xrightarrow{\sim} F(- \otimes \psi_\CM^-(e)) \xrightarrow{\Id_F \circ (\eta_{\CW,e}^1)^{-1}} F(\psi_\CM^-(e) \otimes -) \xrightarrow{\sim} \psi_\CM^-(e) \otimes F(-)$, 

	$\tau_{e,F}^2:\psi_\CM^-(e) \otimes F(-) \xrightarrow{\eta_{\CW,e}^2 \circ \Id_F} F(-) \otimes \psi_\CM^+(e) \xrightarrow{\sim} F(- \otimes \psi_\CM^+(e)) \xrightarrow{\Id_F \circ (\eta_{\CW,e}^2)^{-1}} F(\psi_\CM^-(e) \otimes -) \xrightarrow{\sim} \psi_\CM^-(e) \otimes F(-)$.

	Let $\tau_{e,F}^3=(\tau_{e,F}^1)^{-1} \circ \tau_{e,F}^2$. 	
	Then, we have $\tau_{e,F}^3:\psi_\CM^-(e) \otimes F(-) \xrightarrow{\eta_{\CW,e}^2 \circ \Id_F} F(-) \otimes \psi_\CM^+(e) \xrightarrow{\sim} F(- \otimes \psi_\CM^+(e)) \xrightarrow{\Id_F \circ (\eta_{\CW,e}^2)^{-1}} F(\psi_\CM^-(e) \otimes -) \xrightarrow{\Id_F \circ \eta_{\CW,e}^1} F(- \otimes \psi_\CM^-(e)) \xrightarrow{\sim} F(-) \otimes \psi_\CM^-(e) \xrightarrow{(\eta_{\CW,e}^1)^{-1} \circ \Id_F} \psi_\CM^-(e) \otimes F(-)$.
	
	Then, let $\sigma_{e,F}=\{ F(-) \otimes \psi_\CM^-(e) \xrightarrow{\Id_F \circ (\eta_{\CW,e}^1)^{-1}} \psi_\CM^-(e) \otimes F(-) \xrightarrow{\tau_{e,F}^3} \psi_\CM^-(e) \otimes F(-) \xrightarrow{\Id_F \circ \eta_{\CW,e}^1} F(-) \otimes \psi_\CM^-(e) \}$
	
	$=\{ F(-) \otimes \psi_\CM^-(e) \xrightarrow{\Id_{F(-)} \otimes \eta_\CM} F(-) \otimes \psi_\CM^+(e) \xrightarrow{\sim} F(- \otimes \psi_\CM^+(e)) \xrightarrow{F(\Id_- \otimes \eta_\CM^{-1})} F(- \otimes \psi_\CM^-(e)) \xrightarrow{\sim} F(-) \otimes \psi_\CM^-(e) \}$.
	
	Note that the condition $\tau_{e,F}^1 = \tau_{e,F}^2 = \Id_{e\odot F}$ is equivalent to $\tau_{e,F}^1 = \tau_{e,F}^3 = \Id_{e\odot F}$, while $\tau_{e,F}^3 = \Id_{e\odot F}$ is equivalent to $\sigma_{e,F} = \Id_{e\odot F}$. 
	Also note that under the forgetful functor (forgetting the $\CE$-module structure) $\Fun_{\CE|\CM}(\CM,\CM) \to \Fun_{\CM^\rev}(\CM,\CM) \xrightarrow{\simeq} \CM$, $\sigma_{e,F}$ is carried to the original evident $\CE$-module braiding on $\CM$.

\begin{proof}[Proof of Proposition\ref{prop:mfb-cl}]
	Firstly, we prove the sufficiency. Note that $\CM$ is an $\CE \boxtimes \CE$-module, where the action $\odot:\CE \boxtimes \CE \times \CM \to \CM$, is defined by the formula  $(e_1 \boxtimes e_2) \odot m = \psi_\CM^-(e_1) \otimes m \otimes \psi_\CM^+(e_2)$, for $e_1,e_2 \in \CE$ and $m \in \CM$. Then, we have the following composed equivalence:
	\begin{equation}
	\begin{aligned}
	\Fun_{\CE \boxtimes_{\CE \boxtimes \CE} \CM^\rev}(\CM,\CM) \ &\simeq (\CE \boxtimes \CE) \boxdot_{\CE \boxtimes \CE} \Fun_{\CE|\CM}(\CM,\CM) \\
	&\simeq \CE \boxdot_\CE^2 (\CE \boxdot_\CE^1 \Fun_{\CE|\CM}(\CM,\CM)) \\
	&\simeq \CE \boxdot_\CE^3 (\CE \boxdot_\CE^1 \Fun_{\CE|\CM}(\CM,\CM)) \\
	&\simeq \CE \boxdot_\CE \Fun_{\CM^\rev}(\CM,\CM) \\
	&\simeq \CE \boxdot_\CE \CM \\
	\end{aligned}
	\end{equation}
	where we have used Corollary \ref{cor:expand} both in the first step and in the fourth step. And the second step follows definition, and the third step is due to the analysis of the $\CE$-module braidings on $\Fun_{\CE|\CM}(\CM,\CM)$ before. 
	
	Then, we get that $\CE \boxtimes_{\CE \boxtimes \CE} \CM$ and $\CE \boxdot_\CE \CM$ are Morita equivalent, which leads to the following composed equivalence:
	$
	\FZ(\CE \boxdot_\CE \CM) \simeq \FZ(\CE \boxtimes_{\CE \boxtimes \CE} \CM) \simeq \FZ(\CE) \boxdot_{\CE \boxtimes \CE} \FZ(\CM) \simeq \FZ(\CE) \boxdot_{\CE \boxtimes \CE} (\bar\CC \boxtimes \CD) \simeq \bar\CC \boxdot_\CE \FZ(\CE) \boxdot_\CE \CD \simeq \bar\CC \boxdot_\CE \CD
	$.
	
	Next, we prove the necessity.

	Note that $\FZ(\CE\boxdot_\CE\CM)$ is nondegenerate. Then, by Lemma \ref{lem:non-dege}, the equivalence $\bar\CC\boxdot_\CE\CD \simeq \FZ(\CE\boxdot_\CE\CM)$ implies both $\CC$ and $\CD$ are nondegenerate.

	According to \cite[Corollary 3.26]{DMNO}, any braided monoidal functor between nondegenerate braided fusion categories is fully faithful. 	
	Thus, we have the braided monoidal functor $\psi_\CM:\bar\CC \boxtimes\CD \to \FZ(\CM)$ is fully faithful. 
	Next, according to \cite[Theorem 3.13]{DGNO}, we can assume $\FZ(\CM) \simeq \bar\CC \boxtimes \CD \boxtimes \CB$ where $\CB$ is the centralizer of $\bar\CC \boxtimes \CD$ in $\FZ(\CM)$. 
	Then, we see $\CM$ as a closed multi-fusion $\CC$-$\CD\boxtimes\CB$ bimodule containing $\CE$. Using the sufficiency, we attain $\bar\CC \boxdot_\CE \CD \boxtimes \CB \simeq \FZ(\CE \boxdot_\CE \CM)$, which forces $\CB \simeq \bk$.
\end{proof}

\begin{cor} \label{cor:nde-rpt-nde}
	Let $\CC_1,\CC_2,\CD_1,\CD_2$ be nondegenerate braided fusion categories containing $\CE$, and let ${}_{\CC_1}\CM_{\CC_2},{}_{\CD_1}\CN_{\CD_2}$ be closed multi-fusion bimodules containing $\CE$. Then $\CM\boxdot_\CE\CN$ is a closed multi-fusion $\CC_1\boxdot_\CE\CD_1$-$\CC_2\boxdot_\CE\CD_2$-bimodule. In particular, $\CC\boxdot_\CE\CD$ is a nondegenerate braided fusion category containing $\CE$.
\end{cor}

\begin{proof}
	We have the composed equivalence $(\bar\CC_1 \boxdot_\CE \bar\CD_1) \boxdot_\CE (\CC_2 \boxdot_\CE \CD_2) \simeq \FZ(\CE \boxdot_\CE \CM) \boxdot_\CE \FZ(\CE \boxdot_\CE \CN) \simeq \FZ((\CE \boxdot_\CE \CM) \boxtimes_\CE (\CE \boxdot_\CE \CN)) \simeq \FZ(\CE \boxdot_\CE (\CM \boxdot_\CE \CN))$. This leads to the conclusion.
\end{proof}

\begin{rem}
	Another way to show $\CC\boxdot_\CE\CD$ is nondegenerate for nondegenerate braided fusion categories $\CC$ and $\CD$ is using Proposition \ref{prop:rtp-localm} (the equivalence $\CC\boxdot_\CE\CD \simeq \RMod_{L_\CE}^0(\CC \boxtimes \CD)$) and \cite[Corollary 3.30]{DMNO} (the braided fusion category $\RMod_A^0(\CB)$ is nondegenerate for a condensable algebra $A$ in a braided fusion category $\CB$).
\end{rem}

	We introduce two categories $\mfus_{/\CE}$ and $\bfus^\cl_\CE$ as follows:
\begin{enumerate}
	\item The category $\mfus_{/\CE}$ of multi-fusion categories fully faithful over $\CE$ with the equivalence classes of nonzero semisimple bimodules over $\CE$ as morphisms.
	\item The category $\bfus^\cl_\CE$ of nondegenerate braided fusion categories fully faithful containing $\CE$ with the equivalence classes of closed multi-fusion bimodules containing $\CE$ as morphisms.
\end{enumerate}
	Note that $\bfus^\cl_\CE$ are well-defined due to Corollary \ref{cor:nde-rpt-nde}. Both categories are symmetric monoidal categories under $\boxtimes_\CE$ and $\boxdot_\CE$, with tensor units $\CE$ and $\FZ(\CE)$, respectively.

\begin{thm} \label{thm:fully-faithful_E}
	The symmetry enriched center functor from Theorem \ref{thm:functorial} restricts to a fully faithful symmetric monoidal functor $\FZ: \mfus_{/\CE} \to \bfus^\cl_\CE$.
\end{thm}

\begin{proof}
	According to Theorem \ref{thm:fully-faithful} (the case "over $k$"), $\FZ$ here is a well-defined faithful functor. Moreover, given multi-fusion categories $\CC,\CD$ fully faithful over $\CE$ and a closed multi-fusion $\FZ(\CC)$-$\FZ(\CD)$-bimodule $\CN$ containing $\CE$, there exists a nonzero semisimple $\CC$-$\CD$-bimodule $\CM$ such that $\CN\simeq\Fun_{\CC|\CD}(\CM,\CM)$ as multi-fusion $\FZ(\CC)$-$\FZ(\CD)$-bimodules.
	
	Let $\CW=\Fun_{\CC|\CD}(\CM,\CM)$. Choose a right exact $k$-linear monoidal equivalence $F:\CN\to\CW$ and an isomorphism $\theta$ between the braided monoidal functors $\FZ(F) \circ \psi_\CN: \overline{\FZ(\CC)}\boxtimes\FZ(\CD) \to \FZ(\CN) \to \FZ(\CW)$ and $\psi_\CW: \overline{\FZ(\CC)}\boxtimes\FZ(\CD) \to \FZ(\CW)$. We have an induced isomorphism $\eta_{\CW,e}$ between the monoidal functors $\psi_\CW^-:\CE\to\CW$ and $\psi_\CW^+:\CE\to\CW$, by the following commutative diagram:
	$$
	\xymatrix@C=48pt{
		F(\psi_\CN^-(e)) \ar[r]^-{F(\eta_{\CN,e})} \ar[d]_\theta & F(\psi_\CN^+(e)) \ar[d]^\theta \\
		\psi_\CW^-(e) \ar[r]^-{\eta_{\CW,e}} & \psi_\CW^+(e) \\
	}
	$$
	for $e \in \CE$.
	The isomorphism $\eta_{\CW,e}$ promotes $\CN$ to a left $\CC\boxtimes_\CE\CD^\rev$-module. It is clear from the construction that $\CN$ is equivalent to $\CW$ as multi-fusion $\FZ(\CC)$-$\FZ(\CD)$-bimodules containing $\CE$. This shows that $\FZ$ is full, as desired.
\end{proof}

	For any multi-fusion category $\CC$ over $\CE$, we can take a decomposition $\CC \simeq \CC_0 \oplus \CC_1$ of multi-fusion categories, where $\CC_0$ is fully faithful over $\CE$. 
	Combining Theorem \ref{thm:fully-faithful_E} with Proposition \ref{prop:ZC-ZD-bmeE}, we obtain the following corollary.
	
\begin{cor} \label{cor:Me-Ze_E}
	Two multi-fusion categories $\CC$ and $\CD$ over $\CE$ are Morita equivalent over $\CE$ if and only if $\FZ(\CC) \simeq^{br}_\CE \FZ(\CD)$, i.e. $\FZ(\CC)$ is braided monoidally equivalent to $\FZ(\CD)$ respecting $\CE$.
\end{cor}

	Given a multi-fusion category $\CC$ over $\CE$, we denote the group of the equivalence classes of semisimple invertible $\CC$-$\CC$-bimodules over $\CE$ by $\mathrm{BrPic}_{/\CE}(\CC)$ (a subcategory of $\mathrm{BrPic}(\CC)$), and denote the group of the isomorphism classes of braided auto-equivalences of $\FZ(\CC)$ respecting $\CE$ by $\mathrm{Aut}_\CE^{br}(\FZ(\CC))$ (a subcategory of $\mathrm{Aut}^{br}(\FZ(\CC))$). 
	The following result is a consequence of Proposition \ref{prop:ZC-ZD-bmeE}, Example \ref{exam:ZC-ZD-bimcE} and Theorem \ref{thm:fully-faithful_E}.

\begin{cor} \label{cor:BP-Aut_E}
	Let $\CC$ be a multi-fusion category fully faithful over $\CE$.
	We have $\mathrm{BrPic}_{/\CE}(\CC) \simeq \mathrm{Aut}_\CE^{br}(\FZ(\CC))$ as groups.
\end{cor}

	Etingof, Nikshych and Ostrik proved above two results for the special case $\CE=\bk$ (see Corollary \ref{cor:Me-Ze}, \ref{cor:BP-Aut}) in \cite{ENO2,ENO3}. Kong and Zheng also made it by using Theorem \ref{thm:fully-faithful} in \cite{KZ1}. Here, we generalize their results to the case with symmetry enriched.

%%%%%%%%%%%%%%%%%%%%%%%%%%%%%%%%%%%%%%%%%%%%%%%%%%%%%%%%

\section{Topological orders} \label{sec:TO}

\subsection{Topological orders without symmetry} \label{sec:TOnE}
In this subsection, we give the physical meaning of Theorem \ref{thm:fully-faithful} (the case without symmetry).

In an anomaly-free/anomalous 1+1D topological order, particle-like excitations can be fused, thus such a topological order can in fact be modeled by a fusion category; In an anomaly-free 2+1D topological order, excitations can be not only fused but also braided, thus such a topological order can be modeled by a braided fusion category. (In some texts, the "unitarity" may also be refered to, but here we just drop it and consider the general case.) The vacuum sector corresponds to the tensor unit.

Consider a codimension-1 defect next to an anomaly-free 1+1D topological order. A particle-like excitation in the topological order can be moved to the defect and fused with the excitations in it. This implies such a defect can be described by a semisimple module category over a fusion category. Similarly, for the 2+1D case, the defect should be described by a fusion module over a braided fusion category. Moreover, a wall between two anomaly-free topological orders can be described by a bimodule category. (See the following pictures.)

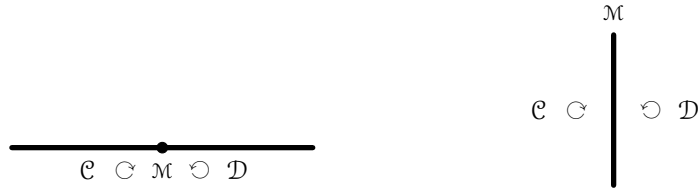
\begin{center} 
	\begin{tikzpicture}
	\draw [line width=2pt,cap=round] (-5,1)-- (-3,1);
	\draw [line width=2pt,cap=round] (-3,1)-- (-1,1);
	\draw [fill=black] (-3,1) circle (2pt);
	\draw (-4,0.7) node {$\CC$};
	\draw (-3.5,0.7) node {\rotatebox{-90}{$\circlearrowright$}};
	\draw (-3,0.7) node [font=\footnotesize] {$\CM$};
	\draw (-2.5,0.7) node {\rotatebox{90}{$\circlearrowleft$}};
	\draw (-2,0.7) node {$\CD$};

	\draw [line width=2pt,cap=round] (3,0.5)-- (3,2.5);
	\draw (2,1.5) node {$\CC$};
	\draw (2.5,1.5) node {\rotatebox{-90}{$\circlearrowright$}};
	\draw (3,2.8) node [font=\footnotesize] {$\CM$};
	\draw (3.5,1.5) node {\rotatebox{90}{$\circlearrowleft$}};
	\draw (4,1.5) node {$\CD$};
	\end{tikzpicture}
\end{center}

Also, the Deligne's tensor product of two (braided) fusion categories formalizes the stacking of two topological orders. 

In an anomaly-free 2+1D topological order, excitations should be able to detect each other via double braidings (see \textit{Self-detection hypothesis} in \cite{KWZ2} ). In particular, a local excitation must have trivial double braidings with all excitations, thus must be the vacuum or one of its direct sums. This amounts to say that the braided fusion category formed by excitations in an anomaly-free 2+1D topological order must be nondegenerate.

The \textit{Unique-bulk hypothesis} (see \cite{KWZ1,KWZ2}) states that for any given $n$D potentially anomalous topological order $\CC_n$, there is a unique anomaly-free $n+1$D topological order, called the \textit{bulk} of $\CC_n $, such that $\CC_n$ can be realized as its gapped boundary. In fact, the universal property that the "bulk" satisfies makes it nothing but the "center" in mathematics.

Then, look at the following picture of topological orders:

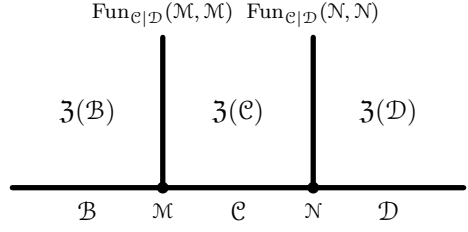
\begin{center} 
	\begin{tikzpicture} 
	\draw [line width=2pt,cap=round] (-3,1)-- (-1,1);
	\draw [line width=2pt,cap=round] (-1,1)-- (-1,3);
	\draw [line width=2pt,cap=round] (-1,1)-- (1,1);
	\draw [line width=2pt,cap=round] (1,1)-- (1,3);
	\draw [line width=2pt,cap=round] (1,1)-- (3,1);
	
	\draw [fill=black] (-1,1) circle (2pt);
	\draw [fill=black] (1,1) circle (2pt);
	
	\draw (-2,2) node {$\FZ(\CB)$};
	\draw (-2,0.7) node {$\CB$};
	
	\draw (0,2) node {$\FZ(\CC)$};
	\draw (0,0.7) node {$\CC$};

	\draw (2,2) node {$\FZ(\CD)$};
	\draw (2,0.7) node {$\CD$};
	
	\draw (-1,0.7) node [font=\footnotesize] {$\CM$};
	\draw (-1,3.3) node [font=\footnotesize] {$\Fun_{\CC|\CD}(\CM,\CM)$};
	
	\draw (1,0.7) node [font=\footnotesize] {$\CN$};
	\draw (1,3.3) node [font=\footnotesize] {$\Fun_{\CC|\CD}(\CN,\CN)$};
	\end{tikzpicture}
\end{center}
It is exactly the physical meaning of Theorem \ref{thm:fully-faithful}.

Similar analyses and mathematical descriptions of topological orders can be found in different texts, such as \cite{KWZ1,KWZ2,KZ3}

\subsection{Symmetry enriched topological orders} \label{sec:SET}
	In this section, we aim to briefly explain the motivation of our work and give the physical meaning of definitions and theorems in Section \ref{sec:SEC} and \ref{sec:SE-center}.

	In fact, our work is inspired by the study on those topological orders with an onsite symmetry, which are usually called symmetry enriched topological (SET) orders. It has been illustrated in \cite{LKW1} that the symmetry can be described by a symmetric fusion category $\CE$ over a field $k$. 
	
	For an anomaly-free/anomalous 1+1D SET order, it is clear the symmetry $\CE$ should have an action on its excitations, thus such a SET order should be modeled by a fusion category over $\CE$.
	
	For an anomaly-free 2+1D SET order, its particle-like topological excitations are desired to be described by a nondegenerate braided fusion category over $\CE$ similarly. However, this data does not fully characterize the SET order. Different from no-symmetry cases, the bulk excitations do not uniquely fix the associated topological orders up to $E_8$ quantum Hall states (see \cite{CGLW,Ki} for examples). One way to complete the data is to categorically gauge the symmetry (see \cite{LKW1,KLWZZ}). As a consequence, an anomaly-free 2+1D SET order can be fully characterized by a triple $(\CE,\CC,\CM)$, where $\CC$ is a unitary modular tensor category over $\CE$ and $\CM$ is a minimal modular extension of $\CC$. Note that $\CM$ is a fusion category containing $\CE$, and $\CC$ is the centralizer of $\CE$ in $\CM$.
	
	Consider a codimension-1 wall between two anomaly-free SET orders. It is clear that the excitations in the bulk still form a bimodule category. However, in the system, the symmetry $\CE$ should be transparent. Therefore, if we want to characterize a wall between two anomaly-free SET orders, we need to use a bimodule category equipped with an additional relation (see the following pictures). Thus we introduce the notion of a bimodule over $\CE$ (see Definition \ref{def:bim_E} and Remark \ref{rem:bim_E}) and that of a monoidal bimodule containing $\CE$ (see Definition \ref{def:mbimcE}).
	
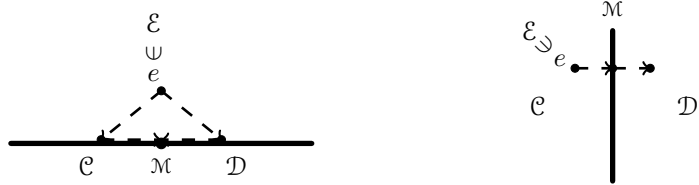
\begin{center}
\begin{tikzpicture}
	\draw [line width=2pt,cap=round] (-5,1)-- (-3,1);
	\draw [line width=2pt,cap=round] (-3,1)-- (-1,1);
	\draw [fill=black] (-3,1) circle (2pt);
	
	\draw (-4,0.7) node {$\CC$};
	\draw (-3,0.7) node [font=\footnotesize] {$\CM$};
	\draw (-2,0.7) node {$\CD$};
	
	\draw (-3.1,1.9) node {$e$};
	\draw (-3.1,2.2) node {\rotatebox{90}{$\in$}};
	\draw (-3.1,2.6) node {$\CE$};
	
	\draw [fill=black] (-3,1.7) circle (1.5pt);
	\draw [fill=black] (-3.8,1.05) circle (1.5pt);
	\draw [fill=black] (-2.2,1.05) circle (1.5pt);
	
	\draw [->,line width=1pt,dash pattern=on 5pt off 5pt] (-3,1.7) -- (-3.8,1.05);
	\draw [->,line width=1pt,dash pattern=on 5pt off 5pt] (-3.8,1.05) -- (-3,1.05);
	\draw [->,line width=1pt,dash pattern=on 5pt off 5pt] (-3,1.7) -- (-2.2,1.05);
	\draw [->,line width=1pt,dash pattern=on 5pt off 5pt] (-2.2,1.05) -- (-3,1.05);

	\draw [line width=2pt,cap=round] (3,0.5)-- (3,2.5);
	\draw (2,1.5) node {$\CC$};
	\draw (3,2.8) node [font=\footnotesize] {$\CM$};
	\draw (4,1.5) node {$\CD$};
	
	\draw (2.3,2.1) node {$e$};
	\draw (2.1,2.3) node {\rotatebox{135}{$\in$}};
	\draw (1.9,2.5) node {$\CE$};
	
	\draw [fill=black] (2.5,2) circle (1.5pt);
	\draw [fill=black] (3,2) circle (1.5pt);
	\draw [fill=black] (3.5,2) circle (1.5pt);
	
	\draw [->,line width=1pt,dash pattern=on 5pt off 5pt] (2.5,2) -- (3,2);
	\draw [->,line width=1pt,dash pattern=on 5pt off 5pt] (3,2) -- (3.5,2);
\end{tikzpicture}
\end{center}	
	
	Also due to the transparentness of $\CE$, we use the Deligne's tensor product over $\CE$ to characterize the stacking of anomaly-free 1+1D SET orders. 
	
	When considering the stacking of anomaly-free 2+1D SET orders, things may be complex. After gauging the symmetry, there are particle-like excitations which have nontrivial double braidings with the excitations in $\CE$. By the intuition from physics, we form a new hypothesis: 
	
	\textit{Two excitations from different SET orders after gauging the symmetry can be stacked together to get a new one if and only if they induce the same double braidings with all excitations in $\CE$. }
	
	Thus we introduce a new tensor product---the relative tensor product over $\CE$ (see Definition \ref{def:rtp}) by involving $\CE$-module braidings (see Definition \ref{def:br}).
	
	In fact, when we use a fusion bimodule containing $\CE$ to describe a domain wall between anomaly-free 2+1D SET orders after gauging the symmetry, we will find there are also $\CE$-module braidings on it (see the following picture). 	
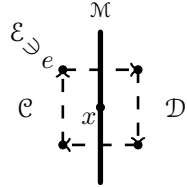
\begin{center}
	\begin{tikzpicture}
	\draw [line width=2pt,cap=round] (3,0.5)-- (3,2.5);
	\draw (2,1.5) node {$\CC$};
	\draw (3,2.8) node [font=\footnotesize] {$\CM$};
	\draw (4,1.5) node {$\CD$};
	
	\draw (2.3,2.1) node {$e$};
	\draw (2.1,2.3) node {\rotatebox{135}{$\in$}};
	\draw (1.9,2.5) node {$\CE$};
	
	\draw [fill=black] (2.5,2) circle (1.5pt);
	\draw [fill=black] (3.5,2) circle (1.5pt);
	\draw [fill=black] (3.5,1) circle (1.5pt);
	\draw [fill=black] (2.5,1) circle (1.5pt);
	
	\draw (2.85,1.35) node {$x$};
	\draw [fill=black] (3,1.5) circle (1.5pt);
	
	\draw [->,line width=1pt,dash pattern=on 5pt off 5pt] (2.5,2) -- (3.5,2);
	\draw [->,line width=1pt,dash pattern=on 5pt off 5pt] (3.5,2) -- (3.5,1);
	\draw [->,line width=1pt,dash pattern=on 5pt off 5pt] (3.5,1) -- (2.5,1);
	\draw [->,line width=1pt,dash pattern=on 5pt off 5pt] (2.5,1) -- (2.5,2);
	\end{tikzpicture}
\end{center}	
	And to describe the stacking of two multi-fusion bimodules containing $\CE$, we also need to use the relative tensor product over $\CE$. 
	
	Then, the following picture of SET orders exactly gives the physical meaning of Theorem \ref{thm:fully-faithful_E}.
	
\begin{center} 
\begin{tikzpicture} 
	\draw [line width=2pt,cap=round] (-3,1)-- (-1,1);
	\draw [line width=2pt,cap=round] (-1,1)-- (-1,3);
	\draw [line width=2pt,cap=round] (-1,1)-- (1,1);
	\draw [line width=2pt,cap=round] (1,1)-- (1,3);
	\draw [line width=2pt,cap=round] (1,1)-- (3,1);
	
	\draw [fill=black] (-1,1) circle (2pt);
	\draw [fill=black] (1,1) circle (2pt);
	
	\draw (-2.5,2.5) node {$\CE \hookrightarrow$};	
	\draw (-2,2) node {$\FZ(\CB)$};
	\draw (-2,0.7) node {$\CB$};
	
	\draw (-0.5,2.5) node {$\CE \hookrightarrow$};
	\draw (0,2) node {$\FZ(\CC)$};
	\draw (0,0.7) node {$\CC$};
	
	\draw (1.5,2.5) node {$\CE \hookrightarrow$};
	\draw (2,2) node {$\FZ(\CD)$};
	\draw (2,0.7) node {$\CD$};
	
	\draw (-1,0.7) node [font=\footnotesize] {$\CM$};
	\draw (-1,3.3) node [font=\footnotesize] {$\Fun_{\CC|\CD}(\CM,\CM)$};
	
	\draw (1,0.7) node [font=\footnotesize] {$\CN$};
	\draw (1,3.3) node [font=\footnotesize] {$\Fun_{\CC|\CD}(\CN,\CN)$};
\end{tikzpicture}
\end{center}

%%%%%%%%%%%%%%%%%%%%%%%%%%%%%%%%%%%%%%%%%%%%%%%%%%%%%%%%

\end{document}